\documentclass[a4paper,11pt,reqno]{amsart}
\usepackage[utf8]{inputenc}
\usepackage[T1]{fontenc}
\usepackage{latexsym,amsmath,amsfonts,amscd,amssymb,amsthm,mathrsfs}
\usepackage{pifont}
\usepackage{tabu}
\usepackage{eucal}
\usepackage{mathabx}
\usepackage{enumerate}
\usepackage[all]{xy}
\usepackage{color}
\definecolor{cobalt}{RGB}{61,89,171}
\usepackage[colorlinks,citecolor=cobalt,linkcolor=cobalt,urlcolor=cobalt,pdfpagemode=UseNone,backref = page]{hyperref}

\theoremstyle{plain}
\newtheorem{theorem}{Theorem}[section]
\newtheorem{proposition}[theorem]{Proposition}
\newtheorem{lemma}[theorem]{Lemma}

\theoremstyle{definition}
\newtheorem{definition}[theorem]{Def{}inition}
\newtheorem{example}[theorem]{Example}
\theoremstyle{remark}
\newtheorem{remark}[theorem]{Remark}

\newcommand{\PL}{\mathrm{PL}}
\newcommand{\bC}{\mathbb{C}}

\newcommand{\bR}{\mathbb{R}}
\newcommand{\bQ}{\mathbb{Q}}
\newcommand{\bZ}{\mathbb{Z}}

\newcommand{\cC}{\mathcal{C}}

\newcommand{\cF}{\mathcal{F}}
\newcommand{\cH}{\mathcal{H}}
\newcommand{\cL}{\mathcal{L}}

\newcommand{\fg}{\mathfrak{g}}
\newcommand{\fh}{\mathfrak{h}}

\newcommand{\ox}{\otimes}
\newcommand{\Lie}[1]{\mathrm{#1}}
\newcommand{\la}{\langle}
\newcommand{\ra}{\rangle}

\begin{document}

\title[Formality and Lefschetz property]{Formality and the Lefschetz property in symplectic and cosymplectic geometry}

\author[G. Bazzoni]{Giovanni Bazzoni}
\address{Fakult\"{a}t f\"{u}r Mathematik, Universit\"{a}t Bielefeld, Postfach 100301, D-33501 Bielefeld}
\email{gbazzoni@math.uni-bielefeld.de}

\author[M. Fern\'andez]{Marisa Fern\'{a}ndez}
\address{Universidad del Pa\'{\i}s Vasco,
Facultad de Ciencia y Tecnolog\'{\i}a, Departamento de Matem\'aticas,
Apartado 644, 48080 Bilbao, Spain}
\email{marisa.fernandez@ehu.es}

\author[V. Mu\~{n}oz]{Vicente Mu\~{n}oz}
\address{Facultad de Matem\'aticas, Universidad Complutense de Madrid, Plaza de Ciencias
3, 28040 Madrid, Spain}
\address{Instituto de Ciencias Matem\'aticas (CSIC-UAM-UC3M-UCM),
C/ Nicol\'as Cabrera 15, 28049 Madrid, Spain}
\email{vicente.munoz@mat.ucm.es}


\begin{abstract}
We review topological properties of K\"ahler and symplectic manifolds, and of
their odd-dimensional counterparts, coK\"ahler and cosymplectic manifolds. 
We focus on formality, Lefschetz property and parity of Betti numbers, 
also distinguishing the simply-connected case (in the K\"ahler/symplectic situation)
and the $b_1=1$ case (in the coK\"ahler/cosymplectic situation).
\end{abstract}

\maketitle
%
%

\date{}






\section{Introduction}

Symplectic geometry is the study of symplectic manifolds, that is, smooth manifolds endowed with a 2-form which is closed and non-degenerate. Important examples of symplectic manifolds are 
$\bR^{2n}=\bR^n\times(\bR^n)^*$ and, more generally, the cotangent bundle $T^*M$ of any smooth manifold $M$, which is endowed with a canonical symplectic structure. Cotangent bundles are especially 
important in classical mechanics, where they arise ase phase spaces of a classical physical system (see \cite{Arnold}). In case the physical system has symmetries, one can perform a reduction 
procedure on the phase space; this often produces compact examples of symplectic manifolds (see \cite{Marsden_Weinstein}). Therefore, the study of compact symplectic manifolds is also relevant. In 
the 
last decades, starting with the first questions on the global nature of symplectic manifolds and the foundational work of Gromov (see \cite{Gromov_Pseudo}), 
symplectic geometry has become a branch of geometry which is interesting \emph{per se}. Standard references for symplectic geometry include \cite{Audin2,Cannas,McD_Salamon}. For more advanced 
applications, including connections with Mirror Symmetry, we refer to \cite{McD_Sal_J_hol,Seidel}.

K\"ahler manifolds are special examples of symplectic manifolds. Recall that a Hermitian manifold is K\"ahler if its K\"ahler form is closed (see Section \ref{intro:symp} for all the relevant 
definition). In particular, symplectic geometry receives many imputs from algebraic geometry. Indeed, the complex projective space $\bC P^n$ is endowed with a K\"ahler structure, which is inherited 
by all projective varieties.

For a long time, the only known examples of symplectic manifolds came in fact from K\"ahler and algebraic geometry. It was only in 1976 that Thurston gave the first example of a symplectic manifold 
which carries no K\"ahler metric (see \cite{Thurston}). Since then, the quest for examples of symplectic manifolds which do not carry K\"ahler metrics has been a very active area of research in 
symplectic geometry, whose main contribution has been the introduction of new techniques for constructing symplectic manifolds 
(see \cite{Auroux,CoFG,Donaldson,FM,Gompf,Gromov_Pseudo,Gromov,McDuff} as well as Section \ref{FLP_S_SC} for an account).

Cosymplectic geometry was first studied by Libermann in \cite{Libermann}. Blair and the Japanese school tracking back to Sasaki have studied cosymplectic geometry in the setting of almost contact 
metric structures (see \cite{Blair} and the references therein). To some extent, cosymplectic geometry can be viewed as the odd-dimensional counterpart of symplectic geometry. Indeed, cosymplectic 
manifolds are the right geometric framework for time-dependent Hamiltonian mechanics. The main example of a cosymplectic manifold is the product of a symplectic manifold with the real line, or 
a circle. There exist however cosymplectic manifolds which are not products (see Example \ref{non_product}). Very recently, cosymplectic manifolds have appeared in the study of a 
special type of Poisson structures, namely $b$-symplectic structures (see \cite{GMP,GMP1}), as well as a special class of foliated manifolds (see \cite{MT}). We refer to \cite{CM_DN_Y} for a recent 
survey on cosymplectic geometry.

Examples of cosymplectic manifolds come from coK\"ahler geometry. CoK\"ahler structures are normal almost contact metric structures $(\phi,\eta,\xi,g)$ for which the tensor $\phi$ is parallel with 
respect to the Levi-Civita connection (see Section \ref{intro:cosymp} for the relevant definitions). The main example is provided by the product of a K\"ahler manifold and the real line (or the 
circle). We point out here as well that there exist compact coK\"ahler manifolds which are not the product of a compact K\"ahler manifold and a circle (see \cite{BO}).

Compact K\"ahler and coK\"ahler manifolds satisfy quite stringent topological properties, which we have collected in Theorems \ref{Topology_Kahler} and \ref{Topology_coKahler} below. One way to 
produce an example of a symplectic manifold which does not carry K\"ahler metrics (resp. of a cosymplectic manifold which does not carry coK\"ahler metrics) consists in constructing a symplectic 
manifold (resp. a cosymplectic manifold) which violates some of these topological properties.

In this paper we study two of these topological properties, which manifest both in the K\"ahler and in the coK\"ahler case, namely the Lefschetz property and the formality of the rational homotopy 
type. The interplay between them has been extensively investigated in the symplectic case, yet nothing has been said in the cosymplectic case. 

In the symplectic framework, our starting point is the paper \cite{IRTU}, in which the authors collected the known (at the time) examples of symplectic manifolds which are, for instance, formal, but 
do not satisfy the Lefschetz property. Some questions remained open in \cite{IRTU}. Due to recent work of many authors (among them we quote Bock \cite{Bock} and Cavalcanti \cite{Cavalc1}), we are 
able 
to answer such questions, completing the picture started in \cite{IRTU}.

Formality is a topological property and it makes perfectly sense to ask whether a cosymplectic manifold is formal. The Lefschetz property, however, turns out to be much more delicate to deal with in 
the odd-dimensional setting. For this, we will restrict to the class of K-cosymplectic structures, i.e. almost coK\"ahler structures for which the Reeb field is Killing (see Section 
\ref{intro:cosymp}). It turns out that the product of a compact almost K\"ahler manifold and a circle is always K-cosymplectic. We will study the relationship between formality and the Lefschetz 
property on product K-cosymplectic manifolds.

\section{Preliminaries}

We collect in this section the basics in symplectic and cosymplectic geometry. We will also do a quick review of rational homotopy theory and formality. Finally, since many of the examples we 
provide are nilmanifolds and solvmanifolds, we shall discuss the main properties of these important classes of manifolds.

\subsection{Symplectic geometry}\label{intro:symp}

A \emph{symplectic manifold} is a pair $(M,\omega)$ where $M$ is a smooth manifold and $\omega$ is a closed, non-degenerate 2-form. Examples of symplectic manifolds are the vector space $\bR^{2n}$, 
the cylinder, the complex projective space $\bC P^n$ and, more generally, coadjoint orbits of semisimple Lie groups.

Recall that an almost Hermitian structure $(g,J)$ on an even-dimensional manifold $M$ consists of a Riemannian metric $g$ and an almost complex structure $J$. Such a structure is called \emph{almost 
K\"ahler} if the K\"ahler form $\omega\in\Omega^2(M)$, defined by $\omega(\cdot,\cdot)=g(\cdot, J\cdot)$, is closed; it is called \emph{K\"ahler} if, in addition, the almost complex structure $J$ is 
integrable. Notice that, if $(M,g,J)$ is an almost K\"ahler manifold, then $M$, together with the K\"ahler form $\omega$, is a symplectic manifold. 

\subsection{Cosymplectic geometry}\label{intro:cosymp}

A \emph{cosymplectic manifold}, as defined by Libermann \cite{Libermann}, is a triple $(M^{2n+1},\eta,\omega)$, where $\eta$ is a 1-form, $\omega$ is a 2-form, both are closed, and 
$\eta\wedge\omega^n\neq 0$. Note that a cosymplectic manifold is necessarily orientable. The product of a compact symplectic manifold and a circle has a natural cosymplectic structure.

An \emph{almost contact structure} on a manifold $M$ is a 4-tuple $(\phi,\eta,\xi,g)$, where
\begin{itemize}
\item a 1-form $\eta$ and a vector field $\xi$, the Reeb field, such that $\eta(\xi)=1$;
\item $\phi\colon TM\to TM$ is a tensor which satisfies $\phi^2=\mathrm{Id}-\eta\ox\xi$;
\item $g$ is a Riemannian metric, which is compatible with $\phi$, namely, $g(\phi\cdot,\phi\cdot)=g(\cdot,\cdot)-\eta(\cdot)\eta(\cdot)$.
\end{itemize}
The K\"ahler form $\omega$ of an almost contact metric structure $(\phi,\eta,\xi,g)$ is $\omega\in\Omega^2(M)$, defined by $\omega(\cdot,\cdot)=g(\cdot, \phi\cdot)$. The almost contact metric 
structure is 
\begin{itemize}
 \item \emph{almost coK\"ahler} if $d\eta=0=d\omega$;
 \item \emph{K-cosymplectic} if it is almost coK\"ahler and the Reeb field is Killing;
 \item \emph{coK\"ahler} if it is almost coK\"ahler and $N_\phi=0$, where $N_\phi$ is the Nijenhuis torsion of $\phi$ (see \cite{Blair}).
\end{itemize}
K-cosymplectic structures have been recently introduced by the first author and Goertsches in \cite{BG}. If $(\phi,\eta,\xi,g)$ is an almost coK\"ahler structure on $M$, then $M$ together 
with $\eta$ and the K\"ahler form $\omega$ is a cosymplectic manifold.

\begin{remark}
In the original works (see for instance \cite{Blair,CdLM}), the term \emph{cosymplectic} was used to denote what is called today coK\"ahler. The terminology coK\"ahler was introduced by Li in 
\cite{Li}, and seems to have become standard since then.
\end{remark}

\subsection{Formality}

The de Rham theorem asserts that the \emph{real} cohomological information of a smooth manifold $M$ can be recovered from the analysis of the smooth forms $\Omega^*(M)$. Apart from the additive 
structure, $\Omega^*(M)$ is endowed with two other operators:
\begin{itemize}
\item the de Rham differential $d\colon\Omega^k(M)\to\Omega^{k+1}(M)$;
\item the wedge product $\wedge\colon\Omega^k(M)\ox\Omega^{\ell}(M)\to\Omega^{k+\ell}(M)$;
\end{itemize}
these interact through the Leibnitz rule, which says that $d$ is a (graded) derivation with respect to the wedge product, i.e. $d(\alpha\wedge\beta)=d\alpha\wedge\beta+(-1)^{\deg\alpha}\alpha\wedge 
d\beta$.

In the seminal paper \cite{Sullivan}, and following previous work of Quillen, Sullivan showed that the infinitesimal nature of the wedge product could be used to extract a great deal of homotopical 
information of a space $X$ from the analysis of \emph{piecewise linear} differential forms $\Omega^*_{\PL}(X)$ on $X$. More precisely, from $\Omega^*_{\PL}(X)$ one can recover not only the rational 
cohomology $H^*(X;\bQ)$ of $X$, but also its rational homotopy groups $\pi_k(X)\ox\bQ$. For this to work, $X$ has to be a \emph{nilpotent} space. This means that $\pi_1(X)$ is a nilpotent group whose 
action on higher homotopy groups is nilpotent. Here and in the sequel, by \emph{space} we mean a CW-complex of finite type. We refer to \cite{FHT,FOT,Griffiths_Morgan} for all the results we quote in 
rational homotopy theory.

\begin{definition}
Let $\mathbf{k}$ be a field of zero characteristic. A \emph{commutative differential graded algebra over $\mathbf{k}$} ($\mathbf{k}$-cdga for short) is a graded vector space $A=\oplus_{n\geq 0}A^n$ 
together with a product $\cdot$ which is commutative in the graded sense, i.e. $x\cdot y=(-1)^{|x||y|}y\cdot x$, where $|x|$ is the degree of a homogeneous element $x$, and a $\mathbf{k}$-linear map 
$d\colon A^{n}\to A^{n+1}$ such that $d^2=0$, which is a graded derivation with respect to $\cdot$, i.e.
\begin{equation}\label{Leibnitz}
d(x\cdot y)=dx\cdot y+(-1)^{|x|}x\cdot dy.
\end{equation}
\end{definition}

Given two cdga's $(A,d)$ and $(A',d')$, a \emph{morphism} is a map $\varphi\colon A\to A'$ which preserves the degree and such that $d'\circ\varphi=\varphi\circ d$. A morphism of cgda's which induces 
an isomorphism in cohomology is a \emph{quasi isomorphism}. 

\begin{example} 
Here are a few examples of cdga's.
\begin{enumerate}
\item[(a)] Given a cdga $(A,d)$, its cohomology is a cgda with trivial differential.
\item[(b)] The de Rham algebra of a smooth manifold, endowed with the wedge product and the exterior differential, is an $\bR$-cdga.
\item[(c)] Piecewise linear differential forms $(\Omega^*_{\PL}(X),d)$ on a space $X$ are a $\bQ$-cdga.
\item[(d)] Given a Lie algebra $\fg$ defined over $\mathbf{k}$, the Chevalley-Eilenberg complex $(\bigwedge\fg^*,d)$ is a $\mathbf{k}$-cdga. Given a basis $\{e_1,\ldots,e_n\}$ of $\fg$ and its dual 
basis $\{e^1,\ldots, e^n\}$ of $\fg^*$, the differential of $e^{k}\in\fg^*$ is defined by $de^k(e_i,e_j)=-e^k([e_i,e_j])$, then extended to $\bigwedge\fg^*$ by imposing \eqref{Leibnitz}. That $d$ 
squares to zero is equivalent to the Jacobi identity in $\fg$.
\end{enumerate}
\end{example}

A very important class of examples of cdga's is provided by Sullivan minimal algebras:
\begin{definition}\label{minimality}
A \emph{Sullivan minimal algebra} is a cdga $(A,d)=(\bigwedge V,d)$ where
\begin{itemize}
\item $\bigwedge V$ is the free commutative algebra generated by a graded vector space $V=\oplus_{n\geq 0}V^n$;
\item there exists a basis $\{x_\tau\}_{\tau\in I}$ of generators of $V$, for some well-ordered index set $I$, such that $|x_\nu|\leq |x_\mu|$ for $\nu<\mu$ and $dx_\mu$ is expressed in term of the 
$x_\nu$ with $\nu<\mu$. In particular, $d$ has no linear part.
\end{itemize}
\end{definition}

The reason why Sullivan minimal algebras are important is given by the following result:

\begin{theorem}
Let $(A,d)$ be a $\mathbf{k}$-cdga, where $\mathrm{char}(\mathbf{k})\neq 0$. Then there exist a Sullivan minimal algebra $(\bigwedge V,d)$ and a morphism $\varphi\colon(\bigwedge V,d)\to(A,d)$ which 
induces 
an isomorphism on cohomology. $(\bigwedge V,d)$, which is unique up to automorphisms, is the minimal model of $(A,d)$.
\end{theorem}

\begin{definition}
Let $X$ be a nilpotent space. The \emph{minimal model} of $X$ is the minimal model of the cdga $(\Omega^*_{\PL}(X),d)$. It is a $\bQ$-cdga, usually denoted by $(\bigwedge V_X,d)$.
\end{definition}

\begin{remark}
When $M$ is a smooth manifold, its \emph{real} minimal model is the minimal model of the de Rham algebra $(\Omega^*(M),d)$.
\end{remark}

Recall that the rationalization of a space $X$ is a rational space $X_\bQ$ (i.e. a space whose homotopy groups are rational vector spaces) together with a map $f\colon X\to X_\bQ$ inducing 
isomorphisms $\pi_k(X)\ox\bQ\cong\pi_k(X_\bQ)$. Two spaces $X$ and $Y$ have the same rational homotopy type if their rationalizations $X_\bQ$ and $Y_\bQ$ have the same homotopy type. Sullivan 
constructed a 1-1 correspondence between nilpotent rational spaces and isomorphism classes of Sullivan minimal algebras over $\bQ$, given by
\[
X\longleftrightarrow (\bigwedge V_X,d).
\]

In this sense, one can study rational homotopy types algebraically.

\begin{definition}\label{minimal}
Let $(A,d)$ be a cdga and let $(\bigwedge V,d)$ be its minimal model. $(A,d)$ is \emph{formal} if there exists a quasi isomorphism $(\bigwedge V,d)\to (H^*(A),0)$. A space $X$ is formal if its minimal 
model 
$(\bigwedge V_X,d)$ is.
\end{definition}

Since, by definition, the cohomology of $(\bigwedge V_X,d)$ is precisely $(H^*(X;\bQ),0)$, a space $X$ is formal if its minimal model (hence its rational homotopy type) is determined by its rational 
cohomology. There are many examples of formal spaces, among them K\"ahler and coK\"ahler manifolds, symmetric spaces, H-spaces.

\begin{remark}

In \cite{Kotschick}, Kotschick called a manifold {\em geometrically formal} if it carries a Riemannian metric for which all wedge 
products of harmonic forms are harmonic. One sees easily that a geometrically formal manifold is formal. Indeed, consider the map $(H^*(M), 0)\to (\Omega^*(M),d)$ which assigns to 
each cohomology class its unique harmonic representative (harmonic with respect to the metric which makes $M$ geometrically formal). Such map is then a morphism of cdga's, and is clearly a quasi  isomorphism. By general theory of minimal models (see for instance \cite[Chapter 1]{Oprea_Tralle}), one gets a quasi isomorphism $(\bigwedge V_M,d)\to (H^*(M), 0)$, hence $M$ is formal. There  are, however, examples of formal manifolds which are not geometrically formal (see \cite{KoTe}). Such examples are generalised symmetric spaces of compact simple Lie groups. It is unclear whether there exist examples of formal, not geometrically formal compact homogeneous non-symmetric spaces. Geometric formality influences the topology of the underlying manifold (see \cite{Ornea_Pilca}).
\end{remark}

Massey products are an obstruction to formality. We describe here triple Massey products and refer to \cite{FM2,Oprea_Tralle} for their higher order analogue. Let $M$ be a manifold and let $a_i\in 
H^{p_i}(M;\bR)$, $1\leq i\leq 3$, be three cohomology classes such that $a_1\cup a_2=0=a_2\cup a_3$. Take forms $\alpha_i$ on $M$ with $[\alpha_i]=a_i$ and write $\alpha_1\wedge\alpha_2=d\sigma$, 
$\alpha_2\wedge\alpha_3=d\tau$. The Massey product of these classes is
\[
 \langle a_1,a_2,a_3\rangle=[\alpha_1\wedge\tau+(-1)^{p_1+1}\sigma\wedge\alpha_3]\in \frac{H^{p_1+p_2+p_3-1}(M;\bR)}{a_1\cup H^{p_2+p_3-1}(M;\bR)+H^{p_1+p_2-1}(M;\bR)\cup a_3}\, .
\]

It was proven in \cite{DGMS} that all Massey products vanish on a formal manifold.

In the present paper we will also deal with spaces which are not nilpotent, for instance solvmanifolds (see Section \ref{nilm_solv} below). When it comes to them, we shall also address the question 
of whether they are formal spaces or not. When we ask such a question, we simply mean to ask whether the minimal model is a formal cdga, in the sense of Definition \ref{minimal}.

\subsection{\textit{s}-formality} \label{subsec:s-formal}

In \cite{FM2}, the second and third authors introduced the notion of
$s$-formality, which is a suitable weakening of the notion of 
formality, and prove that for compact oriented manifolds the weaker
notion implies the stronger one. First note the following result of
\cite{DGMS} which gives a characterization of formality.

\begin{theorem}\label{criterio1}
A minimal model $(A,d)$ is formal if and only if we
can write $A=\bigwedge V$ and the space $V$ decomposes as a
direct sum $V= C\oplus N$ with $d(C) = 0$, $d$ is injective on $N$
and such that every closed element in the ideal $I(N)$ generated
by $N$ in $\bigwedge V$ is exact.
\end{theorem}

In \cite{FM2} we weaken the condition of {formality} as follows.

\begin{definition}\label{primera}
We say that a minimal model $(A,d)$ is {\it
$s$-formal} if we can write $A=\bigwedge V$ such that for each $i\leq s$
the space $V^i$ of generators of degree $i$ decomposes as a direct
sum $V^i=C^i\oplus N^i$, where the spaces $C^i$ and $N^i$ satisfy
the following three conditions:
\begin{enumerate}
\item $d(C^i) = 0$,
\item the differential map $d\colon N^i\rightarrow \bigwedge V$ is
injective,
\item any closed element in the ideal
$I(\bigoplus\limits_{i\leq s} N^i)$, generated by
$\bigoplus\limits_{i\leq s} N^i$ in $\bigwedge
(\bigoplus\limits_{i\leq s} V^i)$, is exact in $\bigwedge V$.
\end{enumerate}
\end{definition}

A connected manifold is $s$-formal if its minimal model is $s$-formal.
The main result of~\cite{FM2} is:

\begin{theorem} \label{criterio2}
Let $M$ be a connected and orientable compact differentiable
manifold of dimension $2n$, or $(2n-1)$. Then $M$ is formal if and
only if it is $(n-1)$-formal.
\end{theorem}

\subsection{Nilmanifolds and solvmanifolds}\label{nilm_solv}

For us, nilmanifolds and solvmanifolds are compact homogeneous spaces of nilpotent and solvable Lie groups. General references are \cite{Auslander,Oprea_Tralle}.

\begin{definition}
 A \emph{(compact) nilmanifold} is the quotient of a simply connected nilpotent Lie group $G$ by a lattice $\Gamma$.
\end{definition}

\begin{definition}
 A \emph{(compact) solvmanifold} is the quotient of a simply connected solvable Lie group $G$ by a lattice $\Gamma$.
\end{definition}

Notice that every nilpotent group is solvable, hence every nilmanifold is a solvmanifold, but the converse is not true. A solvable group $G$ is completely solvable if the adjoint representation on 
$\fg$ has only real eigenvalues. Every nilpotent Lie group is completely solvable.

A connected, simply connected nilpotent Lie group is diffeomorphic to $\bR^n$ 
for some $n$; the diffeomorphism is given by the exponential map. Also, a simply connected solvable Lie group is homeomorphic to $\bR^n$ for some $n$. Hence a solvmanifold $S=\Gamma\backslash G$ (in 
particular, a nilmanifold) is an aspherical space with $\pi_1(S)=\Gamma$.

According to a theorem of Mal'\v cev (see \cite{Mal}), a simply connected nilpotent Lie group $G$ admits a lattice $\Gamma$ if and only if there exists a basis of $\fg$ such that the structure 
constants are rational numbers. So far, a statement of this flavor is not known for simply connected solvable Lie groups. A necessary condition for a Lie group to admit a compact quotient is 
unimodularity. The construction of lattices in solvable Lie groups is an active area of research (see for instance \cite{Bock,Console_Macri}).

From the cohomological point of view, completely solvable solvmanifolds are particularly simple:

\begin{theorem}[Nomizu, \cite{Nomizu} and Hattori, \cite{Hattori}]\label{Nomizu_Hattori_1}
 Let $S=\Gamma\backslash G$ be a completely solvable solvmanifold and let $(\bigwedge\fg^*,d)$ be the Chevalley-Eilenberg complex of $\fg$. Then the inclusion $(\bigwedge\fg^*,d)\hookrightarrow 
(\Omega^*(S),d)$ is a quasi isomorphism.
\end{theorem}

Nilmanifolds are particularly tractable with rational homotopy theoretical methods.

\begin{theorem}[Hasegawa, \cite{Hasegawa}]
 Let $N=\Gamma\backslash G$ be a nilmanifold. Then the Chevalley-Eilenberg complex $(\bigwedge\fg^*,d)$ of $\fg$ is the minimal model of $N$.
\end{theorem}

We point out that, since $G$ admits a lattice by hypothesis, $(\bigwedge\fg^*,d)$ is a $\bQ$-cdga. Hasegawa also proved:
\begin{theorem}[Hasegawa, \cite{Hasegawa}]\label{formal_nilmanifolds}
A nilmanifold $N$ is formal if and only if it is diffeomorphic to a torus.
\end{theorem}

\begin{remark}
 If $S=\Gamma\backslash G$ is a completely solvable solvmanifold, the result of Hattori implies that 
 $(\bigwedge\fg^*,d)$ is a model for $(\Omega^*(S),d)$, meaning that the inclusion is a quasi 
isomorphism. 
\end{remark}

\begin{example}\label{Heisenberg}
 We give an example of a non-toral nilmanifold. Consider the Heisenberg group
 \[
  H=\left\{\begin{pmatrix}
     1 & y & z\\0 & 1 & x\\0 & 0 & 1
    \end{pmatrix} \ | \ x,y,z\in\bR\right\}
    \]
and the lattice $\Gamma=\{A\in H \ | \ x,y,z\in\bZ\}$. Then $N=\Gamma\backslash H$ is a $3$-dimensional nilmanifold. The Chevalley-Eilenberg complex of $\fh$ is
\[
 (\bigwedge(e^1,e^2,e^3), \ de^3=e^1\wedge e^2).
\]
Since $b_1(N)=2$, $N$ is not diffeomorphic to a torus.

\end{example}

\section{Formality and the Lefschetz property in symplectic geometry}\label{FLP_S}

\subsection{Topology of compact K\"ahler manifolds}

Suppose that $M$ is a compact manifold. The question of whether and how the existence of a K\"ahler metric $g$ on $M$ constrains the topology of $M$ has motivated a great deal of research in 
mathematics. Let us recall the main topological features of a compact K\"ahler manifold.

\begin{theorem}\label{Topology_Kahler}
 Let $(M^{2n},J,g)$ be a compact K\"ahler manifold and let $\omega$ be the K\"ahler form. Then
 \begin{enumerate}[(i)]
 \item the odd Betti numbers $b_{2k-1}(M)$ are even, $1\leq k\leq n$;
 \item the Lefschetz map $L^{n-k}\colon H^k(M;\bR)\to H^{2n-k}(M;\bR)$, $[\alpha]\mapsto [\omega]^{n-k}\wedge[\alpha]$ is an isomorphism, $0\leq k\leq n$;
  \item the rational homotopy type of $M$ is formal. 
 \end{enumerate}
\end{theorem}

$(i)$ and $(ii)$ follow from Hodge-de Rham theory on a compact K\"ahler manifold. Notice that $(ii)$ implies $(i)$. This is clear since, for $1\leq k\leq\lfloor\frac{n}{2}\rfloor$, 
the bilinear map
\[
H^{2k-1}(M;\bR)\times H^{2k-1}(M;\bR)\to H^{2n}(M;\bR)\cong\bR, \quad ([\alpha],[\beta])\mapsto [\alpha]\wedge[\beta]\wedge[\omega]^{n-2k+1}
\]
is a symplectic form on $H^{2k-1}(M;\bR)$. Then one simply applies Poincar\'e duality. For a proof of these facts we refer to \cite{Hu}. 

These properties have proven to be extremely useful in the task of constructing examples of compact symplectic manifolds with no K\"ahler metric (see for instance 
\cite{BFM1,Cavalc1,FM,Gompf,McDuff}). It is natural to ask whether these three properties are related on a compact symplectic manifold. Such a question was first tackled in \cite{IRTU}. The authors 
collected the examples, known at the time, of compact symplectic manifolds which violate some of the properties of Theorem \ref{Topology_Kahler}.
The kind of question we want to answer is

\begin{quote}
 Is there a compact symplectic manifold $M$ which satisfies $(i)$ and $(iii)$ above but not $(ii)$?
\end{quote}

The fundamental group plays a crucial role in setting this question. Indeed, while it is relatively easy to come up with non-simply connected compact symplectic manifolds which are non-formal or do 
not satisfy the Lefschetz property, the same question is harder in the simply connected case. We shall therefore subdivide our examples into the simply connected and the non-simply connected case.

\begin{remark}
There is a fourth topological constraint on the topology of compact K\"ahler manifolds: their fundamental groups are not arbitrary (see \cite{ABCKT}). On the other hand, Gompf showed in 
\cite{Gompf} that every finitely presented group is the fundamental group of a symplectic 4-manifold.
\end{remark}

\subsection{The Lefschetz property on compact symplectic manifolds}

Let $(M^{2n},\omega)$ be a compact symplectic manifold and consider, for $0\leq k\leq n$, the map
\begin{equation}\label{Lefschetz_map_forms}
 L^{n-k}\colon \Omega^k(M)\to \Omega^{2n-k}(M), \qquad \alpha\mapsto\omega^{n-k}\wedge\alpha.
\end{equation}

Clearly \eqref{Lefschetz_map} sends closed (resp. exact) forms to closed (resp. exact) forms, hence it descends to a well defined map
\begin{equation}\label{Lefschetz_map}
 L^{n-k}\colon H^k(M;\bR)\to H^{2n-k}(M;\bR).
\end{equation}

\begin{definition}\label{Lefschetz+Lefschetz_type}
 We say that a symplectic manifold $(M^{2n},\omega)$ satisfies the \emph{Lefschetz property} if the Lefschetz map \eqref{Lefschetz_map} is an isomorphism for $0\leq k\leq n$. $(M,\omega)$ 
is of \emph{Lefschetz type} if \eqref{Lefschetz_map} is an isomorphism for $k=1$.
\end{definition}

The Lefschetz map is related to some other important objects which can defined on a compact symplectic manifold. On the one hand, following work of Koszul \cite{Koszul}, Brylinski (see 
\cite{Brylinski}) and Libermann (see \cite{Libermann_Marle}) showed that, on a compact symplectic manifold $(M,\omega)$, one can define a \emph{symplectic} $\star$-operator, 
$\star\colon\Omega^k(M)\to\Omega^{2n-k}(M)$; from this, one gets a symplectic codifferential $\delta\colon\Omega^k(M)\to\Omega^{k-1}(M)$ defined by $\delta\coloneq (-1)^k(\star d \star)$. A form 
$\alpha\in\Omega^k(M)$ is \emph{symplectically harmonic} if 
\[
d\alpha=0 \quad \mathrm{and} \quad \delta\alpha=0.
\]
Denote by $\Omega^*(M,\omega)$ the space of sympletically harmonic forms on $(M,\omega)$. Clearly,
\[
 \cH^k(M,\omega)\coloneq \frac{\Omega^k(M,\omega)}{\Omega^k(M,\omega)\cap \mathrm{Im}\,d}
\]
is a subspace of $H^k(M;\bR)$. We have the following result, proved independently by Mathieu and Yan (see \cite[Corollary 2]{Mathieu} and \cite[Theorem 0.1]{Yan}):
\begin{theorem}
The inclusion $\cH^k(M,\omega)\hookrightarrow H^k(M;\bR)$ is an 
isomorphism if and only if $(M,\omega)$ has the Lefschetz property.
\end{theorem}

The Lefschetz property appears indirectly in another feature of the study of cohomological properties of symplectic manifolds. Given a symplectic manifold $(M,\omega)$, we consider the 
differential $d$ and the symplectic codifferential $\delta$. The following property is known as \emph{$d\delta$-lemma}:
\begin{equation}
 \mathrm{Im}\,d\cap \ker \delta= \mathrm{Im}\,\delta\cap \ker\, d=\mathrm{Im}\,d\delta.
\end{equation}

Merkulov (see \cite{Merkulov}) related the $d\delta$-lemma with the symplectically harmonic cohomology. More precisely, he proved:
\begin{theorem}
A symplectic manifold $(M,\omega)$ satisfies the $d\delta$-lemma if and only if the inclusion $\cH^k(M,\omega)\hookrightarrow H^k(M;\bR)$ is an 
isomorphism.
\end{theorem}

We mention here that the Lefschetz property has been studied also in the context of almost K\"ahler manifolds. Let $(M,g,J)$ be an almost K\"ahler manifold and let 
$\omega$ be the K\"ahler form. Motivated by the Donaldson ``tamed to compatible'' conjecture (see \cite{Donaldson2}), Li and Zhang considered in \cite{Li_Zhang} the following subspaces of 
$H^2(M;\bR)$:
\begin{itemize}
 \item $H^+_J(M)\coloneq \{\alpha \in H^2(M;\bR) \ | \ J\alpha=\alpha\}$;
 \item $H^-_J(M)\coloneq \{\alpha \in H^2(M;\bR) \ | \ J\alpha=-\alpha\}$.
\end{itemize}

One has $H^+_J(M)+H^-_J(M)\subset H^2(M;\bR)$, but the sum is in general neither direct nor equal to $H^2(M;\bR)$. The almost complex structure $J$ is said to be
\begin{itemize}
 \item $\cC^\infty$-pure if $H^+_J(M)\cap H^-_J(M)=0$;
 \item $\cC^\infty$-full if $H^+_J(M)+H^-_J(M)=H^2(M;\bR)$;
 \item $\cC^\infty$-pure-and-full if $H^+_J(M)\oplus H^-_J(M)=H^2(M;\bR)$.
\end{itemize}

Angella, Tomassini and Zhang proved the following result (compare \cite[Theorem 2.3]{Angella_Tomassini_Zhang}; see also \cite[Chapter 4]{Angella}):

\begin{theorem}
 Let $(M^{2n},g,J)$ be an almost K\"ahler manifold and assume that the Lefschetz map $L^{n-2}\colon \Omega^2(M)\to\Omega^{2n-2}(M)$ takes harmonic forms to harmonic forms. Then, if $J$ is 
$\cC^\infty$-full, it is also $\cC^\infty$-pure.
\end{theorem}

To conclude this section, we quote a result of Benson and Gordon, see \cite[Proof of Theorem A]{Benson_Gordon}.
\begin{theorem}\label{Benson_Gordon}
 Let $(N^{2n},\omega)$ be a compact nilmanifold endowed with a symplectic structure. Assume that $(N,\omega)$ is of Lefschetz type. Then $N$ is diffeomorphic to a torus $T^{2n}$.
\end{theorem}

Hence every symplectic non-toral nilmanifold violates condition $(ii)$ in Theorem \ref{Topology_Kahler} and is therefore non K\"ahler. By Nomizu's theorem \ref{Nomizu_Hattori_1}, a symplectic form on 
a nilmanifold is cohomologous to a left-invariant one. Symplectic structures on nilpotent Lie algebras have been studied in \cite{Guan}. A complete classification of symplectic nilmanifolds up to 
dimension $6$ is available in \cite{Bazzoni_Munoz}. As we observed, nilmanifolds are never simply connected.

\subsection{Formality of compact symplectic manifolds}

As we already pointed out, formality is a property of the rational homotopy type of a space, or manifold. Since every compact K\"ahler manifold is formal according to Theorem \ref{Topology_Kahler} 
(see also \cite{DGMS}), it is reasonable to investigate whether formality also holds for arbitrary symplectic manifolds. In the context of nilmanifolds, a formal symplectic nilmanifold is 
diffeomorphic to a torus by Hasegawa's theorem \ref{formal_nilmanifolds}. Hence, a symplectic non-toral nilmanifold violates condition $(iii)$ in Theorem \ref{Topology_Kahler}. Using nilmanifolds, we 
obtain many examples of non formal symplectic manifolds, albeit non simply connected.

In fact, the construction of a \emph{simply connected} symplectic non formal manifold is a much harder problem. As very often happens in symplectic geometry, one of the main problems is that there 
are relatively few techniques to construct symplectic manifolds. Due to this lack of examples, Lupton and Oprea (see \cite{Lupton_Oprea}) asked whether a \emph{simply connected} compact symplectic 
manifold is formal (this is what they called the \emph{formalising tendency} of a symplectic structure). Nowadays, however, we know many simply connected symplectic non formal manifolds: 

\begin{theorem}\label{SCNF}
 For every $n\geq 4$ there exists a compact, simply connected non formal symplectic manifold of dimension $2n$.
\end{theorem}

Theorem \ref{SCNF} was proven by Babenko and Ta\v\i manov (see \cite[Main Theorem]{Babenko_Taimanov}) for $n\geq 5$ and, later, by the second and third author for $n=4$ (see \cite[Theorem 1.1]{FM}). 
The 8-manifold constructed in \cite{FM} does not have the Lefschetz property.

\begin{remark}\label{Miller}
We point out that, due to a result of Miller (see \cite{FM2,Miller}), a simply connected manifold of dimension $\leq 6$ is automatically formal. Hence Theorem \ref{SCNF} covers all 
possible cases in which such a phenomenon can occur. 
\end{remark}

\subsection{Examples: the non-simply connected case}\label{FLP_S_NSC}

Thanks to recent contributions of many authors, we can fill up Table 1 of \cite{IRTU}.

\begin{table}[h]
\caption{Non simply connected symplectic manifolds}
\begin{center}
{\tabulinesep=1.2mm
\begin{tabu}{cccc}
\hline\hline
 Formality & Lefschetz & Evenness of & \\
  & property & $b_{2k-1}$ &\\
\hline\hline
yes & yes & yes & $T^{2n}$, $M^4$\\
yes & yes & no & impossible\\
yes & no & yes & $\Gamma\backslash BG$, $\Lambda\backslash G_{6.78}\times\Lambda\backslash G_{6.78}$\\
yes & no & no & $\Lambda\backslash G_{6.78}$\\
no & yes & yes & $C(12)\times T^2$\\
no & yes & no & impossible\\
no & no & yes &  $E^4$, $KT\times KT$\\
no & no & no & $KT$\\
\hline
\end{tabu}}
\end{center}
\label{table:1}
\end{table}

Let us describe the manifolds which appear in Table \ref{table:1}.

\begin{itemize}
 \item $T^{2n}$ is the $2n$-dimensional torus, which carries a K\"ahler structure.
 \item $M^4$ is a 4-dimensional compact symplectic completely solvable solvmanifold. According to 
 \cite{FG-2} (see also \cite[Theorem 3.2, page 87]{Oprea_Tralle})
 it does not admit a complex structure, hence it is not 
K\"ahler. This shows that there exist symplectic non-K\"ahler manifolds which \emph{can not be detected} using Properties $(i)$, $(ii)$ or $(iii)$ from Theorem \ref{Topology_Kahler}.
 In dimension $\geq 6$, examples of compact symplectic manifolds which satisfy Properties $(i)$, $(ii)$ and $(iii)$ of Theorem \ref{Topology_Kahler} 
 but do not admit a K\"ahler metric due to their fundamental groups are given in \cite{FMS}.
  \item $\Gamma\backslash BG$ is an 8-dimensional solvmanifold, the quotient of a simply connected completely solvable Lie group $BG$ by a lattice $\Gamma$. Benson and Gordon constructed in \cite{BG} 
a completely solvable non-exact symplectic 8-dimensional Lie algebra $\mathfrak{bg}$ whose Lie algebra cohomology $H^*(\mathfrak{bg};\bR)$ does not satisfy the Lefschetz property. In \cite[Theorem 
1]{SY}, Sawai and Yamada showed that $BG$ admits a lattice $\Gamma$. Since $BG$ is completely solvable, Hattori's Theorem \cite{Hattori} applies and we have $H^*(\Gamma\backslash BG;\bR)\cong 
H^*(\mathfrak{bg};\bR)$. In \cite[Proposition 3.2]{FdLS} the second author, de Le\'on and Saralegui showed that the odd-degree Betti numbers of $\mathfrak{bg}$ are even and that its minimal model is 
formal.
 \item $G_{6.78}$ is a simply connected completely solvable symplectic Lie group and $\Lambda\subset G_{6.78}$ is a lattice; we refer to \cite[Theorem 9.4]{Bock} for a complete description of this 
example.
\item $E^4$ is the compact nilmanifold defined by the equations $de^1 = de^2 = 0$, $de^3 = e^1\wedge e^2$, $de^4 = e^1\wedge e^3$ considered in \cite{FGG}.
 It was the first example of compact symplectic manifold of dimension $4$ which does not admit complex structures. In fact, $E^4$ is non-formal and does not
 satisfy the Lefschetz property, but its odd Betti numbers are even. Thus, by the Enriques-Kodaira classification \cite{Kod}, $E^4$ does not have 
 complex structures. Note that in dimension 6, Iwasawa manifold was the first example of compact symplectic and complex manifold whose odd Betti numbers are even, but not admitting K\"ahler metrics, 
as it is non-formal \cite{FG-1}. Non-K\"ahler compact symplectic manifolds in higher dimension are given for example in \cite{CoFG}.
\item $KT$ is the so-called Kodaira-Thurston manifold. As a symplectic manifold, it was first described by Thurston in \cite{Thurston}; there, he showed that $b_1(KT)=3$, 
 hence $KT$ is not K\"ahler. $KT$ also has the structure of a compact complex surface (a \emph{primary Kodaira surface}, see \cite[Page 197]{BHPVdV}) and was known to Kodaira. 
 Abbena described $KT$ as a nilmanifold in \cite{Abbena}. It is the product of the Heisenberg manifold (see Example \ref{Heisenberg}) and a circle.
\item $C(12)$ is a 12-dimensional simply connected symplectic manifold. We shall describe it in Section \ref{FLP_S_SC}.
\end{itemize}

\begin{remark}
In \cite{IRTU}, the authors distinguish between the simply connected and the \emph{aspherical} case, rather than non simply connected. All the examples in Table 1 are aspherical, except 
for $C(12)\times T^2$. We should point out here that a symplectic manifold $(M,\omega)$ is called \emph{symplectically aspherical} if
\[
 \int_{S^2}f^*\omega=0
\]
for every map $f\colon S^2\to M$. An aspherical symplectic manifold is clearly symplectically aspherical. Symplectically aspherical manifolds 
play an important role in symplectic geometry, see \cite{KRT} for a survey on this topic.
\end{remark}


\begin{remark}
In \cite{Kasuya}, Kasuya studied in detail the relationship between Lefschetz property and formality in the case of aspherical manifolds. He constructed examples of formal symplectic 
solvmanifolds which satisfy the Lefschetz property but carry no K\"ahler metric. Kasuya's examples are completely solvable solvmanifolds. For a study of the Lefschetz property on 6-dimensional 
solvmanifolds, see also \cite{Macri}. By a recent result of Hasegawa \cite{Hasegawa2}, a compact solvmanifold admits a K\"ahler structure if and only if it is a finite quotient of a complex torus 
which has the structure of a complex torus bundle over a complex torus. In particular, a K\"ahler completely solvable solvmanifold is diffeomorphic to a torus.
\end{remark}

\subsection{Examples: the simply connected case}\label{FLP_S_SC}

We collect and illustrate recent results which allow us to fill all the lines in Table 2 of \cite{IRTU}.

\begin{table}[h!]
\caption{Simply connected symplectic manifolds}
\begin{center}
{\tabulinesep=1.2mm
\begin{tabu}{cccc}
\hline\hline
 Formality & Lefschetz & Evenness of & \\
  & property & $b_{2k-1}$ &\\
\hline\hline
yes & yes & yes & $\bC P^n$\\
yes & yes & no & impossible\\
yes & no & yes & $M(6,0,0)$, $N$\\
yes & no & no & $\widetilde{\bC P}{}^7$\\
no & yes & yes & $C(12)$\\
no & yes & no & impossible\\
no & no & yes & $\widetilde{\bC P}{}^5\times\widetilde{\bC P}{}^5$\\
no & no & no & $\widetilde{\bC P}{}^5$\\
\hline
\end{tabu}}
\end{center}
\label{table:2}
\end{table} 

Before describing the content of Table \ref{table:2}, we would like to recall a construction of McDuff (see \cite{McDuff,Oprea_Tralle}) which allows to obtain new symplectic manifolds: the symplectic 
blow-up. Together with the fibre connected sum (see \cite{Gompf} and \cite[Chapter 7]{McD_Salamon}), symplectic fibrations (see \cite[Chapter 6]{McD_Salamon} and the references therein), 
approximately holomorphic techniques of Donaldson (see \cite{Auroux,Donaldson,MPS}) and symplectic resolutions (see \cite{CFM,FM}), it is the most effective technique when it comes to 
constructing new symplectic manifolds.

Let $(M^{2n},\omega)$ be a symplectic manifold. Without loss of generality, we can perturb $\omega$ and assume that it defines an integral cohomology class. By a result of Tischler \cite{Tischler} 
(see also \cite{Gromov}), there exists a symplectic embedding $\imath\colon (M^{2n},\omega)\hookrightarrow (\bC P^{2n+1},\omega_0)$, i.e. $\imath^*\omega_0=\omega$, where $\omega_0$ is the 
Fubini-Study K\"ahler structure of the complex projective space.

Using ideas of Gromov, and generalizing a known construction in algebraic geometry, McDuff defined the notion of \emph{symplectic blow-up} of a symplectic manifold $(X,\sigma)$ along 
a symplectic submanifold $(Y,\tau)$. Since $Y\subset X$ is a symplectic submanifold, the normal bundle $NY$ of $Y$ in $X$ has the structure of a complex vector bundle. The blow-up of $X$ along $Y$, 
replaces a point $y\in Y$ with the projectivization of $N_yY$. This produces a new manifold $\tilde{X}=\mathrm{Bl}_Y X$ and a map $p\colon \tilde{X}\to X$, with the following properties:
\begin{itemize}
 \item \cite[Proposition 2.4]{McDuff} $\pi_1(X)=\pi_1(\tilde{X})$ and $H^*(\tilde{X};\bR)$ fits into a short exact sequence of $\bR$-modules
 \begin{equation}\label{cohomology_blow_up}
  0\rightarrow H^*(X;\bR)\rightarrow H^*(\tilde{X};\bR)\rightarrow A^*\rightarrow 0,
 \end{equation}
where $A^*$ is a free module over $H^*(Y;\bR)$ with one generator in each dimension $2i$, $1\leq i\leq k-1$, where $2k$ is the codimension of $Y$ in $X$.
\item \cite[Proposition 3.7]{McDuff} If $Y$ is compact, $\tilde{X}$ carries a symplectic form $\tilde{\sigma}$ which agrees with $p^*\sigma$ outside a neighborhood of $p^{-1}(Y)$.
\end{itemize}


We also recall the following result of Gompf, which allows to construct compact symplectic manifolds which do not satisfy the Lefschetz property.
\begin{theorem}[\cite{Gompf}, Theorem 7.1]\label{Gompf}
 For any even dimension $n\geq 6$, finitely presented group $G$ and integer $b$ there exists a closed symplectic $n$-manifold $M=M(n,G,b)$ with $\pi_1(M)\cong G$ and $b_i(M)\geq b$ for $2\leq i\leq 
n-2$, such that $M$ does not satisfy the Lefschetz property. Furthermore, if $b_1(G)$ is even, then all odd-degree Betti numbers of $M$ are even.
\end{theorem}

We begin now the description of the manifolds in Table \ref{table:2}.

\begin{itemize}
 \item $\bC P^n$ is the complex projective space, which is known to have a K\"ahler structure.
 \item $M(6,0,0)$ is the manifold constructed by taking $n=6$, $G$ the trivial group and $b=0$ in Theorem \ref{Gompf}. Since it is a simply connected 6-manifold, it is formal by the result of Miller, 
see Remark \ref{Miller}.
 \item $N$ is a $6$-dimensional simply connected (hence formal) symplectic and complex manifold which does not satisfy the Lefschetz property. Such a manifold was constructed by the authors in 
\cite{BFM1}.
 \item We shall construct $\widetilde{\bC P}{}^7$ in Proposition \ref{ynn} below.
 \item $C(12)$ is a simply connected, symplectic $12$-manifold which satisfies the Lefschetz property but has a non-zero triple Massey product. Hence it is non-formal. Such example was constructed by 
Cavalcanti in \cite[Example 4.4]{Cavalc1} using Donaldson's techniques together with the symplectic blow-up.
\item $\widetilde{\bC P}{}^5$ is the symplectic blow-up of $\bC P^5$ along a symplectic embedding of the Kodaira-Thurston manifold $KT$. Historically, this was the first example of a simply connected 
symplectic manifold with no K\"ahler structures. It was constructed by McDuff in \cite{McDuff}.
\end{itemize}

We come now to the description of $\widetilde{\bC P}{}^7$. Let us consider the symplectic $6$-manifold $S=\Lambda\backslash G_{6.78}$ which appeared in Section \ref{FLP_S_NSC}. Recall that $S$ is a 
formal symplectic manifold which does not satisfy the Lefschetz property and has $b_1(S)=1$. According to Tischler's result, we find a symplectic embedding of $(S,\omega)$ in $(\bC P^7,\omega_0)$. 
Now set $\widetilde{\bC P}{}^7\coloneq \mathrm{Bl}_S\bC P^7$.

\begin{proposition}\label{ynn}
 $\widetilde{\bC P}{}^7$ is a simply connected, formal symplectic manifold which does not satisfy the Lefschetz property and has $b_3(\widetilde{\bC P}{}^7)=1$.
\end{proposition}

\begin{proof}
That $\pi_1(\widetilde{\bC P}{}^7)=0$ follows immediately from McDuff result we discussed above. $b_3(\widetilde{\bC P}{}^7)=1$ follows from \eqref{cohomology_blow_up} and the fact that $b_1(S)=1$. 
This immediately implies that $\widetilde{\bC P}{}^7$ does not have the Lefschetz property.

By \cite[Theorem 1.1]{FM4}, $\widetilde{\bC P}{}^7$ is formal. However, let us see this explicitly
by computing its minimal model. 
The $6$-manifold $S=\Lambda\backslash G_{6.78}$ is the quotient of 
the simply connected completely solvable symplectic Lie group $G_{6.78}$ 
of \cite[Theorem 9.4]{Bock} by a lattice $\Lambda\subset G_{6.78}$. The
cohomology of $S$ is computed from the cdga $(\bigwedge (\alpha_1,\alpha_2,\alpha_3,
\alpha_4,\alpha_5,\alpha_6),d)$, with all $\alpha_i$ of degree one, and
\begin{itemize}
 \item $d\alpha_1=\alpha_2\wedge\alpha_5-\alpha_1\wedge\alpha_6$;
 \item $d\alpha_2=\alpha_4\wedge\alpha_5$;
 \item $d\alpha_3=\alpha_2\wedge\alpha_4+\alpha_3\wedge\alpha_6+\alpha_4\wedge\alpha_6$;
 \item $d\alpha_4=\alpha_4\wedge\alpha_6$;
 \item $d\alpha_5=-\alpha_5\wedge\alpha_6$;
 \item $d\alpha_6=0$.
\end{itemize}
The manifold $S$ has a symplectic structure, defined by 
$\omega=\alpha_1\wedge\alpha_4+\alpha_2\wedge\alpha_6+\alpha_3 \wedge\alpha_5$.
The cohomology of $S$ is given by
\begin{itemize}
 \item $H^0(S;\bR)=\langle 1\rangle$;
 \item $H^1(S;\bR)=\langle [\eta]\rangle$;
 \item $H^2(S;\bR)=\langle [\omega]\rangle$;
 \item $H^3(S;\bR)=\langle [\omega\wedge\eta], [\psi]\rangle$;
 \item $H^4(S;\bR)=\langle[\omega^2]\rangle$;
 \item $H^5(S;\bR)=\langle[\omega\wedge \psi]\rangle$;
 \item $H^6(S;\bR)=\langle[\omega^3]\rangle$;
\end{itemize}
where $\eta=\alpha_6$, $\psi=\alpha_1\wedge\alpha_2\wedge\alpha_4-\alpha_1\wedge\alpha_4\wedge
\alpha_6-\alpha_2\wedge\alpha_3\wedge\alpha_5$. Note that 
$\omega^2\wedge \eta= d\gamma$, where $\gamma=2\alpha_1\wedge\alpha_2\wedge\alpha_4\wedge\alpha_5$.

Now consider $(S,\omega) \subset (\bC P^7,\omega_0)$ and the blow-up $\pi\colon\widetilde{\bC P}{}^7=
\mathrm{Bl}_S\bC P^7 \to \bC P^7$. Let $\nu$ be a Thom form for the exceptional divisor
$E\coloneq \pi^{-1}(S)$. This is supported on a tubular neighbourhood $V$ of $E$.
Consider the projections $\pi'\colon V\to E$, $\pi''\colon\pi(V)\to S$ and the composition $\varpi=\pi\circ \pi'=\pi''\circ \pi\colon V\to S$.
Finally, let $\omega_0'$ be a $2$-form on $\bC P^7$
cohomologous to $\omega_0$ but coinciding with $\pi''^*\omega$ on $\pi(V)$.
By \eqref{cohomology_blow_up}, we have the cohomology of $\widetilde{\bC P}{}^7$. 
\begin{itemize}
 \item $H^0(\widetilde{\bC P}{}^7;\bR)=\langle 1\rangle$;
 \item $H^1(\widetilde{\bC P}{}^7;\bR)=0$;
 \item $H^2(\widetilde{\bC P}{}^7;\bR)=\langle [\pi^*\omega_0'], [\nu]\rangle$;
 \item $H^3(\widetilde{\bC P}{}^7;\bR)=\langle [\varpi^*\eta\wedge \nu]\rangle$;
 \item $H^4(\widetilde{\bC P}{}^7;\bR)=\langle [\pi^*\omega_0'^2],
  [\pi^*\omega_0' \wedge\nu], [\nu^2]\rangle$;
 \item $H^5(\widetilde{\bC P}{}^7;\bR)=\langle 
 [\varpi^*(\omega \wedge \eta)\wedge \nu], [\varpi^*\psi\wedge\nu],
 [\varpi^*\eta\wedge \nu^2]\rangle$;
 \item $H^6(\widetilde{\bC P}{}^7;\bR)=\langle [\pi^*\omega_0'^3], 
 [\pi^*\omega_0'^2 \wedge \nu], [\pi^*\omega_0' \wedge \nu^2], [\nu^3]\rangle$;
  \item $H^7(\widetilde{\bC P}{}^7;\bR)=\langle [\varpi^*(\omega\wedge \psi)\wedge\nu],
   [\varpi^*(\omega \wedge \eta)\wedge \nu^2], [\varpi^*\psi\wedge\nu^2],
 [\varpi^*\eta\wedge \nu^3]\rangle$.
\end{itemize}
As in \cite{FM4}, we compute the minimal model 
$\varphi:(\bigwedge V,d) \to (\Omega^*(\widetilde{\bC P}{}^7),d)$. Up to degree $6$, this is
\begin{itemize}
 \item $V^1=0$;
 \item $V^2=C^2=\la x,y\ra$, $\varphi(x)=\pi^*\omega_0'$, $\varphi(y)=\nu$;
 \item $V^3=C^3=\la w\ra$, $\varphi(w)=\varpi^*\eta\wedge \nu$;
 \item $V^4=0$;
 \item $V^5= C^5=\la z\ra$, $\varphi(z)= \varpi^*\psi\wedge \nu$;
 \item $V^6= N^6= \la u\ra$, $du= x^2w$, $\varphi(u)=\varpi^*\gamma\wedge \nu$.
 \end{itemize}
We check that $\widetilde{\bC P}{}^7$ is $6$-formal (see Definition \ref{criterio1}). 
Take an element $\beta\in I(N^6)$ which
is closed, and we have to check that $[\varphi(\beta)]=0$ in $H^*(\widetilde{\bC P}{}^7;\bR)$. 
As $\beta$ can have degree at most $14$, it has to be $\beta= \beta'\cdot w\cdot u$, with
$\beta'\in  \bigwedge (x,y,z)$. But clearly $[\varphi(w\cdot u)]=0$ since 
$[\eta \wedge\gamma] =0 \in H^5(S;\bR)$.
By Theorem \ref{criterio2}, $\widetilde{\bC P}{}^7$ is formal.
\end{proof}

%
%
%
%
\begin{remark}
It seems quite hard to find a compact simply connected symplectic manifold which is formal and satisfies the Lefschetz property but is not K\"ahler.
\end{remark}

\begin{remark}
The manifold $N$ which appears in Table \ref{table:2}, line 3, carries a symplectic structure and a complex structure but no K\"ahler metric. For a discussion on manifolds which are 
simultaneously complex and symplectic but not K\"ahler, we refer to \cite[Section 1]{BFM1}.
\end{remark}

\section{Formality and the Lefschetz property in cosymplectic geometry}\label{FLP_CS}

A compact cosymplectic manifold $(M^{2n+1},\eta,\omega)$ is never simply connected. Indeed, consider the 1-form $\eta$; being 
closed, it defines a cohomology class $[\eta]\in H^1(M;\bR)$. If $\eta=df$, with $f\in C^\infty(M)$, by compactness of $M$ there exists $p\in M$ so that $\eta_p=0$. But this contradicts the 
condition $\eta(\xi)\equiv 1$. Alternatively, one can argue using the fact that $\eta\wedge\omega^n$ is a volume form, hence, if $\eta$ were exact, the same would be true for the volume form, which 
is absurd. In this way, one sees that $b_1(M)\geq 1$ on a cosymplectic manifold. This argument allows actually to conclude that $\eta\wedge\omega^k$ is a closed, non 
exact form, for $0\leq k\leq n$. Hence, $b_k(M)\geq 1$ for $0\leq k\leq 2n+1$ on a cosymplectic manifold.

CoK\"ahler manifolds are the odd-dimensional counterpart of K\"ahler manifolds. It is not a big surprise, therefore, that they satisfy as well strong topological conditions.

\begin{theorem}\label{Topology_coKahler}
 Let $(M^{2n+1},\phi,\eta,\xi,g)$ be a compact coK\"ahler manifold and let $\omega$ be the K\"ahler form. Then
 \begin{enumerate}[(i)]
  \item $b_1(M)$ is odd;
  \item the Lefschetz map $L^{n-k}\colon \cH^k(M)\to \cH^{2n+1-k}(M)$, 
  \begin{equation}\label{Lefschetz_map_coK}
   \alpha\mapsto \omega^{n-k+1}\wedge \imath_\xi\alpha+\omega^{n-k}\wedge\eta\wedge\alpha
  \end{equation}
is an isomorphism, $0\leq k\leq n$;
  \item the rational homotopy type of $M$ is formal. 
 \end{enumerate}
\end{theorem}

A proof of this result can be found in \cite{CdLM}; see also \cite{BLO} for a different perspective on how such properties can be deduced from the corresponding properties of compact K\"ahler 
manifolds. Here $\cH^*(M)$ denotes harmonic forms on the Riemannian manifold $(M,g)$; clearly $\cH^*(M)\cong H^*(M;\bR)$. The map \eqref{Lefschetz_map_coK} really sends harmonic forms to harmonic 
forms, as it is proved in \cite{CdLM}; recall that both $\eta$ and $\omega$ are parallel on a coK\"ahler manifold, hence harmonic.

Here as well, we want to answer questions such as

\begin{quote}
 Is there a compact cosymplectic manifold $M$ which satisfies $(i)$ and $(iii)$ above but not $(ii)$?
\end{quote}

Unfortunately, as we shall see in the next Section, the Lefschetz map can not be defined, in general, on arbitrary cosymplectic manifolds. We will identify a certain property, morally equivalent to 
the \emph{Lefschetz type} condition of Definition \ref{Lefschetz+Lefschetz_type}, and address the Lefschetz question in this setting.


\subsection{The Lefschetz property in cosymplectic geometry}\label{LP_C}

Let $(M,\omega)$ be a symplectic manifold. As we remarked above, the Lefschetz map \eqref{Lefschetz_map_forms} sends closed forms to closed forms, hence descends to cohomology, giving 
\eqref{Lefschetz_map}. In particular, the Lefschetz map can be defined on \emph{any} symplectic manifold. Of course, one needs then to use K\"ahler identities to prove that 
\eqref{Lefschetz_map} is an isomorphism in the K\"ahler case and we have seen that there are symplectic manifolds for which \eqref{Lefschetz_map} is not an isomorphism.

The cosymplectic case is much subtler. A first instance is that, differently from what happens in the K\"ahler case, the Lefschetz map on coK\"ahler manifolds is defined only on 
\emph{harmonic} forms. On a cosymplectic manifold, however, there is no metric. Let $(M,\eta,\omega)$ be a compact cosymplectic manifold; if one tries to define a map 
\begin{equation}\label{Lefschetz_map_cos}
 L^{n-k}\colon\Omega^k(M)\to \Omega^{2n+1-k}(M), \quad \alpha\mapsto \omega^{n-k+1}\wedge \imath_\xi\alpha+\omega^{n-k}\wedge\eta\wedge\alpha
\end{equation}
one sees that it \emph{does not send closed forms to closed forms!} Indeed, for $\alpha$ a closed $k$-form,
\begin{equation}\label{closed_non_closed}
 d(\omega^{n-k+1}\wedge \imath_\xi\alpha+\omega^{n-k}\wedge\eta\wedge\alpha)=\omega^{n-k+1}\wedge d(\imath_\xi\alpha),
\end{equation}
which is not zero in general.

\begin{example}\label{non_2_Lefschetz}
 Consider the 5-dimensional Lie algebra $\fg$ with non-trivial brackets
 \[
  [e_1,e_2]=-e_4, \quad [e_1,e_5]=-e_3,
 \]
with respect to a basis $\{e_1,\ldots,e_5\}$. The associated Chevalley-Eilenberg complex is then
 \[
  (\bigwedge(e^1,\ldots,e^5), \ de^3=e^{15}, de^4=e^{12}),
 \]
where $e^{ij}$ is a short-hand for $e^i\wedge e^j$. Set $\eta=e^5$ and $\omega=e^{13}-e^{24}$. Then $(\fg,\eta,\omega)$ is a cosymplectic Lie algebra; the Reeb field is $\xi=e_5$. Let $G$ 
denote the simply connected nilpotent Lie group with Lie algebra $\fg$. Since the Lie algebra $\fg$ is defined over $\bQ$, $G$ contains a lattice $\Gamma$; hence $N\coloneq\Gamma\backslash G$ is a 
compact nilmanifold. The cosymplectic structure on $\fg$ gives a left-invariant cosymplectic structure on $G$, which descends to $N$. Hence 
$(N,\eta,\omega)$ is a cosymplectic nilmanifold. By Nomizu's theorem, $(\bigwedge\fg^*,d)\hookrightarrow(\Omega^*(N),d)$ is a quasi isomorphism. We study the Lefschetz map on 
$(\bigwedge\fg^*,d)$, i.e.
\[
 L^{5-k}\colon \bigwedge\nolimits^k\fg^*\longrightarrow\bigwedge\nolimits^{5-k}\fg^*,
\]
$0\leq k\leq 2$. A computation shows that $L^{4}$ sends closed forms to closed forms. However, the closed 2-form $\alpha=e^{35}$ is sent to $\beta=-e^{234}$, which is not closed, 
since $d\beta=e^{1245}$.
\end{example}

A way to bypass this difficulty would be to work with forms on $M$ which are preserved by the flow of the Reeb field $\xi$, that is, to consider the differential subalgebra 
$\Omega^*_\xi(M)$ of $\Omega^*(M)$, defined by
\[
 \Omega^k_\xi(M)=\{\alpha\in\Omega^k(M) \ | \ \cL_\xi\alpha=0\}.
\]
\begin{lemma}
The Lefschetz map \eqref{Lefschetz_map_cos} restricts to a map 
 \begin{equation}\label{Lefschetz_map_Kcos}
 L^{n-k}\colon\Omega^k_\xi(M)\to \Omega^{2n+1-k}_\xi(M)
\end{equation}
which sends closed forms to closed forms.
\end{lemma}
\begin{proof}
 Assume that $\alpha\in\Omega^k_\xi(M)$; hence $\cL_\xi\alpha=0$ and $d(\imath_\xi\alpha)=-\imath_\xi d\alpha$. We need to prove that $L^{n-k}(\alpha)\in \Omega^{2n+1-k}_\xi(M)$. Recall that 
$\cL_\xi\omega=0=\cL_\xi\eta$ on a cosymplectic manifold. We compute
 \begin{align*}
  \cL_\xi(L^{n-k}(\alpha))&=\cL_\xi(\omega^{n-k+1}\wedge \imath_\xi\alpha)+\cL_\xi(\omega^{n-k}\wedge\eta\wedge\alpha)=\\
  &=\omega^{n-k+1}\wedge\cL_\xi(\imath_\xi\alpha)+\omega^{n-k}\wedge\eta\wedge\cL_\xi\alpha=\\
  &=\omega^{n-k+1}\wedge(d\imath_\xi+\imath_\xi d)(\imath_\xi\alpha)=\\
  &=-\omega^{n-k+1}\wedge \imath_\xi\imath_\xi d\alpha=\\
  &=0.
 \end{align*}
 Also, if $\alpha\in\Omega^k_\xi(M)$ is closed, then \eqref{closed_non_closed} shows that $d(L^{n-k}(\alpha))=\omega^{n-k+1}\wedge d(\imath_\xi\alpha)$. Since $\cL_\xi\alpha=0$, we can switch $d$ and 
$\imath_\xi$, obtaining $d(L^{n-k}(\alpha))=0$.
\end{proof}
Everything works just fine, but there is of course a problem: the differential graded algebra $(\Omega^*_\xi(M),d)\subset (\Omega^*(M),d)$ does not compute, in general, the de Rham cohomology of $M$.
The question arises, whether there exists a class of almost coK\"ahler structures for which the inclusion $(\Omega^*_\xi(M),d)\hookrightarrow (\Omega^*(M),d)$ is a quasi isomorphism, 
i.e. 
for which $H^*_\xi(M)\coloneq H^*(\Omega^*_\xi(M),d)$ satisfies
\[
 H^*_\xi(M)\cong H^*(M;\bR).
\]

Recall that an almost coK\"ahler structure $(\phi,\eta,\xi,g)$ is K-cosymplectic if the Reeb field is Killing. K-cosymplectic structures have been extensively studied in \cite{BG}; the 
inspiration there came from the contact (metric) case, where one defines a K-contact structure as a contact metric structure 
whose Reeb field is Killing.

On a K-cosymplectic manifold $M$, the 1-dimensional distribution defined by $\xi$ integrates to a Riemannian foliation $\cF_\xi$, whose leaf through $x\in M$ is the 
flowline of $\xi$.

\begin{definition}
 The \emph{basic cohomology} of $\cF_\xi$ is the cohomology of the differential subalgebra
 \[
  \Omega^*(M;\cF_\xi)=\{\alpha\in\Omega^*(M) \ | \ \imath_\xi\alpha=0=\imath_\xi d\alpha\}.
 \]
\end{definition}

We denote by $H^*(M;\cF_\xi)$ the basic cohomology. One can think of it as the cohomology of the space of leaves $M/\cF_\xi$. We collect in the following Proposition the relevant 
features of K-cosymplectic structures (compare \cite[Theorem 4.3]{BG}):
\begin{proposition}\label{harmonic}
Let $(\phi,\eta,\xi,g)$ be a K-cosymplectic structure on a compact manifold $M^{2n+1}$. Then
\begin{enumerate}
\item the inclusion $(\Omega^*_\xi(M),d)\hookrightarrow (\Omega^*(M),d)$ induces an isomorphism $H^*_\xi(M)\cong \cH^*(M)$; in particular the cdga $(\Omega^*_\xi(M),d)$ computes the de 
Rham cohomology of $M$;
\item there is a splitting
\begin{equation}\label{splitting}
H^k_\xi(M)=H^k(M;\cF_\xi)\oplus \eta\wedge H^{k-1}(M;\cF_\xi);
\end{equation}
\item $H^{2n}(M;\cF_\xi)\cong\bR$ and $H^{2n+1}(M;\cF_\xi)=0$.
\end{enumerate}
\end{proposition}

Here we interpret $\eta$ as a harmonic $1$-form.
As a consequence, the Lefschetz map \eqref{Lefschetz_map_Kcos} on 
a K-cosymplectic manifold $(M^{2n+1},\phi,\eta,\xi,g)$ descends to a map 
\begin{equation}\label{LKC_C}
 L^{n-k}\colon H^k_\xi(M)\to H^{2n+1-k}_\xi(M), \quad \alpha\mapsto \omega^{n-k+1}\wedge\imath_\xi\alpha+\omega^{n-k}\wedge\eta\wedge\alpha
\end{equation}
which can, or not, be an isomorphism.

From now on, we restrict to K-cosymplectic manifolds in order to study the Lefschetz map.

\begin{definition}
Let $M^{2n+1}$ be a compact manifold endowed with a K-cosymplectic structure $(\phi,\eta,\xi,g)$. We say that 
$M$ has the \emph{Lefschetz property} if \eqref{LKC_C} is an isomorphism for $0\leq k\leq n$. We say that $M$ is of \emph{Lefschetz type} if \eqref{LKC_C} is an isomorphism for $k=1$.
\end{definition}

\begin{proposition}
 Let $M^{2n+1}$ be a compact manifold endowed with a K-cosymplectic structure $(\phi,\eta,\xi,g)$. Assume that $M$ satisfies is of Lefschetz type. Then $b_1(M)$ is odd.
\end{proposition}

\begin{proof}
The splitting \eqref{splitting} tells us that $H^1_\xi(M)=H^1(M;\cF_\xi)\oplus \eta$ and that $H^{2n}_\xi(M)=H^{2n}(M;\cF_\xi)\oplus\eta\wedge H^{2n-1}(M;\cF_\xi)$. By assumption, $L^{n-1}\colon 
H^1_\xi(M)\to H^{2n}_\xi(M)$ is an isomorphism. Its restriction to $H^1(M;\cF_\xi)$ sends $\alpha$ to $\omega^{n-1}\wedge\eta\wedge\alpha\in \eta\wedge H^{2n-1}(M;\cF_\xi)$. In particular, 
$\omega^{n-1}\wedge\alpha\neq 0$. Now consider the bilinear map
\[
 \Psi\colon H^1(M;\cF_\xi)\otimes H^1(M;\cF_\xi)\to H^{2n}(M;\cF_\xi)\cong \bR, \quad (\alpha,\beta)\mapsto \omega^{n-1}\wedge\alpha\wedge\beta.
\]
$\Psi$ is clearly skew-symmetric and non-degenerate. Hence $\dim H^1(M;\cF_\xi)$ is even. Since $b_1(M)=\dim H^1(M;\cF_\xi)+1$, the thesis follows.
\end{proof}

We have some kind of converse to this result:

\begin{proposition}\label{b_1=1}
Suppose $M$ is a compact manifold endowed with a K-cosymplectic structure $(\phi,\eta,\xi,g)$. If $b_1(M)=1$, then $M$ is of Lefschetz type.
\end{proposition}

\begin{proof}
Since $M$ is compact and $b_1(M)=1$, $H^1_\xi(M)=\eta$. The Lefschetz map \eqref{LKC_C} sends $\eta$ to $\omega^n\in H^{2n}_\xi(M)$, which is non-zero.
\end{proof}

\begin{remark}
It is known that $b_1$ is odd on a compact manifold endowed with a coK\"ahler structure; however, nothing can be said on the higher odd-degree Betti numbers, even up to middle dimension. Consider, 
for instance, the manifold $M\times M\times S^1$, where $M$ is the $K3$ surface. This is clearly coK\"ahler and has $b_3=44$. This is why, in Table \ref{table:4} below, the column about 
Betti numbers has been 
removed.
\end{remark}

We show next that K-cosymplectic manifolds abund. Indeed, the following holds (see \cite[Proposition 2.12]{BG}):

\begin{proposition}\label{K-cosymplectic_mapping_torus}
 Let $(K,h,\omega)$ be a compact almost K\"ahler manifold and let $\varphi\colon K\to K$ be a diffeomorphism such that 
$\varphi^*h=h$ and $\varphi^*\omega=\omega$. Then the mapping torus\footnote{Given a topological space $X$ and a homeomorphism $\varphi\colon X\to X$, the \emph{mapping torus} 
$X_\varphi$ is the quotient space $\frac{X\times [0,1]}{(x,0)\sim (\varphi(x),1)}$.} $K_\varphi$ has a natural K-cosymplectic structure.
\end{proposition}

In particular, the product $K\times S^1$ of an almost K\"ahler manifold $(K,h,\omega)$ and a circle has a natural K-cosymplectic structure, with the product 
metric $g=h+d\theta^2$. The cosymplectic structure is given by taking $\xi$ to be the vector tangent to $S^1$ and $\eta$ to be the dual 1-form. For a product K-cosymplectic manifold, 
$M=K\times S^1$, the space of leaves of $\cF_\xi$ is simply $K$, hence \eqref{splitting} becomes
\begin{equation}\label{splitting_2}
 H^k_\xi(M)=H^k(K;\bR)\oplus\eta \wedge H^{k-1}(K;\bR).
\end{equation}

Under this splitting, $\alpha\in H^k_\xi(M)$ can be written as $\alpha_0+\eta\wedge\alpha_1$ with $\alpha_0\in H^k(K;\bR)$ and $\alpha_1\in H^{k-1}(K;\bR)$. We have $\imath_\xi\alpha_j=0$, $j=0,1$, 
and $\eta\wedge\eta\wedge\alpha_1=0$. As a consequence, the Lefschetz map \eqref{LKC_C} sends $\alpha_0\in H^k(K;\bR)\subset H^k_\xi(M)$ to $\omega^{n-k}\wedge\eta\wedge\alpha_0$ and 
$\eta\wedge\alpha_1$ to $\omega^{n-k+1}\wedge\alpha_1$.

We easily obtain the following result:

\begin{proposition}
 Let $(K,h,\omega)$ be a compact almost K\"ahler manifold. Then the compact K-cosymplectic manifold $K\times S^1$ satifies the Lefschetz property \eqref{Lefschetz_map_coK} if and only if 
$(K,\omega)$ satifies the Lefschetz property \eqref{Lefschetz_map}.
\end{proposition}

We focus on a situation in which one could think of studying some analogue of the Lefschetz property on arbitrary cosymplectic manifolds. We make the following definition:

\begin{definition}\label{alg_model}
Let $M$ be a compact connected manifold. We say that the cohomology of $M$ \emph{can be computed from an algebraic model} if there exist a commutative differential graded algebra (cdga) 
$(\bigwedge V_M,d)$, defined over $\bR$, where the elements of $V_M$ have degree $1$, 
and a morphism of cdga's $\varphi\colon(\bigwedge V_M,d)\to (\Omega^*(M),d)$, inducing and isomorphism in cohomology, 
i.e.\ $H^*(\bigwedge V_M,d)\cong H^*(M;\bR)$.
\end{definition}

The usual terminology for a morphism of cdga's, inducing an isomorphism on cohomology, is quasi-isomorphism. It follows from this definition that if $M$ is an $n$-dimensional manifold, then 
$\dim(V_M)=n$, because a generator of the top power of $V_M$ must be sent to the volume form on $M$. Suppose $\dim V_M=n$; then $\bigwedge\nolimits^n V_M$ is one-dimensional, hence a generator 
$\mathrm{vol}\in 
\bigwedge\nolimits^n V_M$ is sent to the volume form on $M$ by $\varphi$. 
By compactness, and using the fact that 
$\varphi$ commutes with the differential, this implies that $\mathrm{vol}$ can not be exact and, as a consequence, that every element $w$ in $\bigwedge\nolimits^{n-1}V_M$ must be closed; otherwise one 
would 
necessarily have $dw=c\mathrm{vol}$, with $c\in\bR^*$, making $\mathrm{vol}$ exact. 

\begin{proposition}\label{Nomizu_Hattori}
 Let $S=\Gamma\backslash G$ be a compact solvmanifold. Then the cohomology of $M$ can be computed from an algebraic model.
\end{proposition}

\begin{proof}
 This follows from Nomizu Theorem \cite{Nomizu} when $S$ is a nilmanifold and from Hattori Theorem \cite{Hattori} when $S$ is a completely solvable solvmanifold. In both cases, the algebraic model 
is the Chevalley-Eilenberg complex $(\bigwedge\fg^*,d)$ of the Lie algebra $\fg$ of $G$. For the non-completely solvable case, one uses a result of Guan (\cite{Guan2}, see also \cite{CF}) to obtain a 
modification of the Lie group $G$, $\tilde{G}$, such that $S=\Gamma\backslash \tilde{G}$ and the Lie algebra cohomology of $\tilde{\fg}$ computes $H^*(S;\bR)$.
\end{proof}

%
%
%
%

%
Let $(M^{2n+1},\eta,\omega)$ be a cosymplectic manifold satisfying the hypotheses of Definition \ref{alg_model}. Let $(\bigwedge V_M,d)$ be the algebraic model and let $\varphi\colon(\bigwedge 
V_M,d)\to(\Omega^*(M),d)$. Since $\varphi$ induces an isomorphism in cohomology, and both $\eta$ and $\omega$ are closed non-exact forms, one can choose cocycles $v\in V_M$ and $w\in 
\bigwedge\nolimits^2 V_M$ 
such that 
$[\varphi(v)]=[\eta]$ and $[\varphi(w)]=[\omega]$. Notice that $v\wedge w^n$ is a generator of 
$\bigwedge\nolimits^{2n+1}V_M$. 
Since $M$ is connected, $v$ is uniquely determined. If $\bar{w}\in \bigwedge\nolimits^2 V_M$ is such 
that $[\bar{w}]=[w]$ in $H^2(\bigwedge V_M)$, then there exists $u\in V_M$ such that $\bar{w}=w+du$. 
Then $v\wedge \bar{w}^n=v\wedge w^n$, since they differ by an exact form and
there are no exact forms in $\bigwedge\nolimits^{2n+1}V_M$. 
The (algebraic) Reeb field $\theta\in V_M^*$ is determined by the 
equation
\[
 \imath_\theta ( v\wedge w^n) =w^n.
\]

Take a closed form $\gamma\in\Omega^1(M)$. We find a cocycle $z\in V_M$ such that 
$[\varphi(z)]=[\gamma]$. Define an algebraic Lefschetz map $V_M\to\bigwedge\nolimits^{2n}V_M$ by
\begin{equation}\label{alg_Lef_map}
 z\mapsto y\coloneq w^{n-1}\wedge v \wedge z+w^n\wedge\imath_\theta z.
\end{equation}
By what we said so far, $y$ is a closed element in $\bigwedge\nolimits^{2n}V_M$. Hence the algebraic Lefschetz 
map sends $1$-cocycles to $2n$-cocycles. Therefore, it descends to cohomology and it is reasonable to ask whether it is an isomorphism there.

\begin{definition}\label{1-Lefschetz}
 Let $(M,\eta,\omega)$ be a compact cosymplectic manifold and assume that the cohomology of $M$ can be computed from an algebraic model. We say 
that $M$ is \emph{$1$-Lefschetz} if \eqref{alg_Lef_map} is an isomorphism.
\end{definition}

In the case of nilmanifolds, we have the following result, which can be seen as the cosymplectic analogue of the Benson-Gordon Theorem \ref{Benson_Gordon}.

\begin{theorem}
 Let $N=\Gamma\backslash\Lie{G}$ be a compact nilmanifold of dimension $2n+1$ endowed with a cosymplectic structure $(\eta,\omega)$. Assume that $N$ is $1$-Lefschetz. Then $N$ is 
diffeomorphic to a torus.
\end{theorem}

\begin{proof}
The Chevalley-Eilenberg complex $(\bigwedge\fg^*,d)$ of the Lie algebra $\fg$ is an algebraic model of the cohomology of $N$, as follows from Proposition \ref{Nomizu_Hattori}; notice that 
$\dim\fg=2n+1$. Let $\varphi\colon(\bigwedge\fg^*,d)\to(\Omega^*(N),d)$ be the quasi-isomorphism. In particular, the cohomology of $N$ can be computed from this algebraic model. In order to prove 
that 
$N$ is diffeomorphic to a torus, it is enough to prove that $d\equiv 0$ in $(\bigwedge\fg^*,d)$ (see the argument in \cite[Theorem 2.2, Chapter 2]{Oprea_Tralle}). Arguing as above, we can assume that 
there are cocycles $v\in\fg^*$ and 
$w\in\bigwedge\nolimits^2\fg^*$ mapping to $\eta$ and $\omega$ respectively 
under $\varphi$. Assume that $d$ is non-zero. Then we can choose a Mal'cev basis of $\fg^*$,
\[
 \fg^*=\la v,x_1,\ldots,x_s,x_{s+1},\ldots,x_{2n}\ra,
\]
with $v=x_0$, $dx_i=0$, $0\leq i\leq s$, and $dx_j$, $s+1\leq j\leq 2n$, 
is a non-zero linear combination of products $x_{k\ell}:=x_k\wedge x_\ell$ with $k,\ell<j$, 
for some $s<2n+1$. Then $w$ can be written as
\begin{equation}\label{sympl_form}
 w=\sum_{0\leq i<j\leq 2n}a_{ij}x_{ij}+z\cdot x_{2n},
\end{equation}
for some coefficients $a_{ij}\in\bR$; we have collected all the summands that contain $x_{2n}$
into $z\in \fg^*$, where $z$ does not contain $x_{2n}$. 
Also, notice that $z$ must be a cocycle. Indeed, when one computes $dw$, which must be zero, the term $dz\cdot x_{2n}$ pops up. But since we have chosen a Mal'cev basis for 
$\fg^*$, the generator $x_{2n}$ can not appear in the differential of any other $x_k$, and this forces $z$ to be a cocycle. We define a derivation $\lambda$ of degree $-1$ on $\fg^*$ by the rule
\[
 \lambda(x_i)=0, \ 0\leq i\leq 2n, \quad \lambda(x_{2n})=1,
\]
and extend it to $(\bigwedge\fg^*,d)$ by forcing the Leibnitz rule. Assume that the algebraic Lefschetz map \eqref{alg_Lef_map} is an isomorphism; the cocycle $z$ is sent to the cocycle 
$w^{n-1}\wedge 
v\wedge z$. For degree reasons we have $y\wedge v\wedge w^n=0$ for every $1$-cocycle. 
Applying $\lambda$ we get
\begin{align*}
 0& = \lambda(y\wedge v\wedge w^n)=\lambda(y)\wedge \eta\wedge w^n-y\wedge\lambda(v)\wedge w^n
  +y\wedge v\wedge \lambda(\omega^n)=\\
  & = n\,y\wedge v\wedge\lambda(w)\wedge w^{n-1}=n\,y\wedge v\wedge z\wedge w^{n-1}\, .
\end{align*}
This must be true for every $1$-cocycle $y$, and $v\wedge z\wedge w^{n-1}$ is non-zero by the Lefschetz-type hypothesis. But this violates Poincar\'e duality. Thus we obtain a contradiction with the 
existence of a non-closed generator of $\fg^*$, since $\lambda$ was defined as 0 on such generators. Hence $d\equiv 0$.
\end{proof}

\begin{remark}
 Recall that a manifold $M^{2n}$ is \emph{cohomologically symplectic} is there is a class $\omega\in H^2(M;\bR)$ such that $\omega^n\neq 0$. Every symplectic manifold is cohomologically symplectic, 
but the converse is not true. Indeed, $\bC P^2\# \bC P^2$ is cohomologically symplectic but admits no almost complex structures, as shown by Audin \cite{Audin}. Hence, it can not be symplectic. The 
nice interplay between geometry and topology on cohomologically symplectic manifolds has been unveiled by Lupton and Oprea (see \cite{Lupton_Oprea_Gottlieb}). It was proven recently by Kasuya (see 
\cite{Kasuya_2}) that cohomologically symplectic \emph{solvmanifolds} are genuinely symplectic. 
Following these ideas, one can call a manifold $M^{2n+1}$ 
\emph{cohomologically cosymplectic} if there exist classes $\eta\in H^1(M;\bR)$ and $\omega\in H^2(M;\bR)$ such that $\eta\wedge\omega^n\neq 0$. By arguing as in \cite{Kasuya_2}, one can show that a 
cohomologically cosymplectic solvmanifold is cosymplectic.
\end{remark}

Example \ref{non_2_Lefschetz} above shows that, even in the case of cosymplectic manifolds whose cohomology can be computed from an algebraic model, the Lefschetz map does not necessarily send 
cocycles of degree $k\geq 2$ to cocycles.

\subsection{Formality in cosymplectic geometry}\label{F_C}

As a consequence of formality of compact K\"ahler manifolds, it was proved in \cite{CdLM} that a compact coK\"ahler manifold is formal. In \cite{BFM}, we constructed examples of 
non-formal cosymplectic manifolds with arbitrary Betti numbers. More precisely, we proved the following result:

\begin{theorem}\label{geography}
 For every pair $(m=2n+1,b)$ with $n,b\geq 1$ there exists a compact non-formal manifold $M$, endowed with a cosymplectic structure, of dimension $m$ and with first Betti number equal to $b$, with 
the exception of the pair $(3,1)$.
\end{theorem}

It is known that an orientable compact manifold of dimension $\leq 4$ with first Betti number equal to 1 is formal, see \cite{FM3}; in other words, the exception $(3,1)$ is not due to the fact that 
we require the manifold to have a cosymplectic structure.

\subsection{Examples}

We let now Sections \ref{LP_C} and \ref{F_C} come together, and fill up two tables similar to those given in Section \ref{FLP_S} in the symplectic case. We have already observed that a compact 
manifold endowed with a cosymplectic structure always has $b_1\geq 1$. Therefore, one can never have simply connected compact cosymplectic manifolds. We will therefore distinguish two cases: 
\begin{itemize}
 \item compact K-cosymplectic manifolds with arbitrary first Betti number;
 \item compact K-cosymplectic manifolds with first Betti number equal to $1$.
\end{itemize}

%
%

Since the product of a compact symplectic manifold and a circle always admits a K-cosymplectic structure, it is enough to take products of the manifolds which appear in Tables \ref{table:1} 
and \ref{table:2} with a circle. The resulting manifolds are collected in Tables \ref{table:3} and \ref{table:4}. For a description of the examples, we refer to the discussion in Sections 
\ref{FLP_S_NSC} and \ref{FLP_S_SC}

\begin{table}[h]
\caption{Product K-cosymplectic manifolds with arbitrary $b_1$}
\begin{center}
{\tabulinesep=1.2mm
\begin{tabu}{cccc}
\hline\hline
 Formality & Lefschetz & $b_1$ odd & \\
  & property & &\\
\hline\hline
yes & yes & yes & $T^{2n+1}$, $M^4\times S^1$\\
yes & yes & no & impossible\\
yes & no & yes & $\Gamma\backslash BG\times S^1$, $\Lambda\backslash G_{6.78}\times\Lambda\backslash G_{6.78}\times S^1$\\
yes & no & no & $\Lambda\backslash G_{6.78}\times S^1$\\
no & yes & yes & $C(12)\times T^3$\\
no & yes & no & impossible\\
no & no & yes & $KT\times KT\times S^1$, $E^4\times S^1$\\
no & no & no & $KT\times S^1$\\
\hline
\end{tabu}}
\end{center}
\label{table:3}
\end{table}



\begin{table}[h]
\caption{Product K-cosymplectic manifolds with $b_1=1$}
\begin{center}
{\tabulinesep=1.2mm
\begin{tabu}{cccc}
\hline\hline
 Formality & Lefschetz property & \\
\hline\hline
yes & yes & $\bC P^n\times S^1$\\
yes & no  & $M(6,0,0)\times S^1$, $N\times S^1$\\
no & yes  & $C(12)\times S^1$\\
no & no   & $\widetilde{\bC P}{}^5\times S^1$\\
\hline
\end{tabu}}
\end{center}
\label{table:4}
\end{table}

We end this section with two further examples: a compact cosymplectic non-formal $5$-manifold with $b_1=1$ which is not the product of a $4$-manifold and a circle and a compact K-cosymplectic 
$7$-manifold which is not coK\"ahler and is not the product of a $6$-manifold and a circle.

\begin{example}\label{non_product}
In \cite{BFM}, we constructed a cosymplectic, non-formal $5$-dimensional solvmanifold $S$ with $b_1(S)=1$. $S$ is the quotient of the completely solvable Lie group $H$, whose elements are matrices
\[
A= \begin{pmatrix}
  e^{-x_5} & 0 & 0 & 0 & 0 & x_1\\
  0 & e^{x_5} & 0 & 0 & 0 & x_2\\
  -x_5e^{-x_5} & 0 & e^{-x_5} & 0 & 0 & x_3\\
  0 & -x_5e^{x_5} & 0 & e^{x_5} & 0 & x_4\\
  0 & 0 & 0 & 0 & 1 & x_5 \\
  0 & 0 & 0 & 0 & 0 & 1
 \end{pmatrix}
\]
with $(x_1,\ldots,x_5)\in\bR^5$, by a lattice $\Gamma$. A global system of coordinates on $H$ is given by $\{X_1,\ldots,X_5\}$ with $X_i(A)=x_i$. A basis of left-invariant $1$-forms on $H$ (which we 
identify with $\fh^*$) is given by
\begin{itemize}
 \item $\alpha_1=e^{X_5}dX_1$;
 \item $\alpha_2=e^{-X_5}dX_2$;
 \item $\alpha_3=X_5e^{X_5}dX_1+e^{X_5}dX_3$;
 \item $\alpha_4=X_5e^{-X_5}dX_2+e^{-X_5}dX_4$;
 \item $\alpha_5=dX_5$.
\end{itemize}

A straightforward computation shows that
\begin{itemize}
 \item $d\alpha_1=-\alpha_1\wedge\alpha_5$;
 \item $d\alpha_2=\alpha_2\wedge\alpha_5$;
 \item $d\alpha_3=-\alpha_1\wedge\alpha_5-\alpha_3\wedge\alpha_5$;
 \item $d\alpha_4=-\alpha_2\wedge\alpha_5+\alpha_4\wedge\alpha_5$;
 \item $d\alpha_5=0$.
\end{itemize}

The manifold $S$ has a cosymplectic structure, defined by taking
\[
 \eta=\alpha_5 \quad \mathrm{and} \quad \omega=\alpha_1\wedge\alpha_4+\alpha_2\wedge\alpha_3.
\]

Thus $(S,\eta,\omega)$ is a compact cosymplectic $5$-dimensional non-formal solvmanifold. Since $H$ is completely solvable, we can apply Hattori's theorem and conclude that $H^*(S;\bR)\cong 
H^*(\fh;\bR)$. 
We compute
\begin{itemize}
 \item $H^0(S;\bR)=\langle 1\rangle$;
 \item $H^1(S;\bR)=\langle[\eta]\rangle$;
 \item $H^2(S;\bR)=\langle[\alpha_1\wedge\alpha_2],[\omega]\rangle$;
 \item $H^3(S;\bR)=\langle[\alpha_3\wedge\alpha_4\wedge\eta],[\omega\wedge\eta]\rangle$;
 \item $H^4(S;\bR)=\langle[\omega^2]\rangle$;
 \item $H^5(S;\bR)=\langle[\eta\wedge\omega^2]\rangle$.
\end{itemize}

Notice that $S$ is $1$-Lefschetz, according to Definition \ref{1-Lefschetz}. We prove that $S$ is not the 
product of a $4$-manifold and a circle. Assume this is the case and write $S=P\times S^1$. We use the product structure of $H^*(S;\bR)\cong H^*(P;\bR)\otimes H^*(S^1;\bR)$. The generator $[\eta]$ 
of $H^1(S^1;\bR)$ is a zero divisor, since $[\eta]\wedge[\alpha_1\wedge\alpha_2]=0$ for $[\alpha_1\wedge\alpha_2]\in H^2(P;\bR)$, and this is absurd.
\end{example}

\begin{example}\label{K-cosymplectic}
 Finally, we give an example of a $7$-dimensional K-cosymplectic solvmanifold without coK\"ahler structures, which is not a product of a $6$-dimensional manifold and a circle. This shows 
that, although we considered \emph{product} K-cosymplectic manifolds, there are more. This example is inspired by the discussion in \cite[Chapter 3, Section 3]{Oprea_Tralle}. We start with the 
$6$-dimensional nilpotent Lie algebra $\fh$ with basis $\{e_1,\ldots,e_6\}$ and non-zero brackets
\[
 [e_1,e_2]=-e_4, \quad [e_1,e_3]=-e_5 \quad \mathrm{and} \quad [e_2,e_3]=-e_6.
\]
Notice that $\fh$ is the Lie algebra $\fh_7$ in the notation of \cite{Salamon} and the Lie algebra $L_{6,4}$ in that of \cite{Bazzoni_Munoz}. The Chevalley-Eilenberg complex is $(\bigwedge \fh^*,d)$ 
with non-zero differential given, in terms of the dual basis $\{e^1,\ldots,e^6\}$, by
\[
 de^4=e^{12}, \quad de^5=e^{13} \quad \mathrm{and} \quad de^6=e^{23}.
\]

We endow $\fh$ with the following almost K\"ahler structure:
\begin{itemize}
 \item the scalar product $h$ which makes $\{e_1,\ldots,e_6\}$ orthonormal;
 \item the symplectic form $\omega=-e^{16}+e^{25}+2e^{34}$.
\end{itemize}
Consider the derivation $D\colon\fh\to\fh$ given, with respect to $\{e_1,\ldots,e_6\}$, by the matrix
\[
D= \begin{pmatrix}
  0 & 1 & 0 & 0 & 0 & 0\\
  -1 & 0 & 0 & 0 & 0 & 0\\
  0 & 0 & 0 & 0 & 0 & 0\\
  0 & 0 & 0 & 0 & 0 & 0\\
  0 & 0 & 0 & 0 & 0 & 1 \\
  0 & 0 & 0 & 0 & -1 & 0
 \end{pmatrix}.
\]

Since $D$ is skew-symmetric, it is an infinitesimal isometry of $(\fh,h)$; furthermore, one can check that $D$ is an infinitesimal symplectic derivation of $(\fh,\omega)$, i.e.,\ $D^t\omega+\omega 
D=0$. We denote by $\fg$ the semi-direct product $\fh\oplus_D\bR$, which is a $7$-dimensional solvable (but not completely solvable) Lie algebra with brackets:
\begin{itemize}
 \item $[e_i,e_j]_\fg=[e_i,e_j]_\fh$, for $1\leq i<j\leq 6$;
 \item $[e_7,e_i]_\fg=D(e_i)$ for $1\leq i\leq 6$; here $e_7$ generates the $\bR$-factor.
\end{itemize}
We see that $\fg$ sits in a short exact sequence $0\to\fh\to\fg\to\bR\to 0$ of Lie algebras and that $\fh\subset\fg$ is an ideal. Define $\varphi_t\coloneq\exp(tD)\in\mathrm{Aut}(\fh)$. Then 
$\varphi_t^*\omega=\omega$ and $\varphi_t^*h=h$. Let $H$ be the unique simply connected nilpotent Lie group with Lie algebra $\fh$ and consider the following diagram:
\[
 \xymatrix{
 H\ar[r]^{\Phi_t} & H\\
 \fh\ar[u]^{\exp}\ar[r]_{\varphi_t}& \fh\ar[u]_{\exp}
 }
\]
where $\Phi_t\coloneq \exp\circ\varphi_t\circ\exp^{-1}$ is a well-defined Lie group automorphism; indeed, $H$ is nilpotent, hence $\exp$ is a global diffeomorphism.
We consider $h$ and $\omega$ as left-invariant objects on $H$. Then $\Phi_t$ is, by construction, an isometry of $(H,h)$ and a symplectomorphism of $(H,\omega)$.

Let $\overline{\Lambda}$ be the lattice in $\fh$ spanned over $\bZ$ by the basis $\{e_1,\ldots,e_6\}$ and set $\Lambda\coloneq\exp(\overline{\Lambda})$. Then $\Lambda\subset H$ is a lattice and 
$N\coloneq\Lambda\backslash H$ is a compact nilmanifold. Notice that $\varphi_{\frac{\pi}{2}}$ preserves $\overline{\Lambda}$, hence $\Phi=\Phi_{\frac{\pi}{2}}$ preserves $\Lambda$ and descends to a 
diffeomorphism of $N$. Since $h$ and $\omega$ are left-invariant, they define a metric and a symplectic structure on $N$. Clearly, $\Phi$ is both an isometry and a symplectomorphism 
of $(N,h,\omega)$. Now consider the solvable Lie group $G=H\rtimes_\Phi\bR$ and the lattice $\Gamma=\Lambda\rtimes_\Phi\bZ$. Notice that $\Gamma$ is a solvable group. We have two short exact sequnces 
of groups, namely
\[
 0\rightarrow\Lambda\rightarrow\Gamma\rightarrow\bZ\rightarrow 0
\]
and
\[
 0\rightarrow H\rightarrow G\rightarrow \bR\rightarrow 0.
\]
Then $S\coloneq \Gamma\backslash G$ is a solvmanifold which can be identified with the mapping torus of the diffeomorphism $\Phi \colon N\to N$. Indeed, the mapping torus bundle 
coincides with the Mostow bundle of $S$, which is $N\to S\to S^1$.

By Proposition \ref{K-cosymplectic_mapping_torus}, $S$ has a natural K-cosymplectic structure. We can describe as a left-invariant K-cosymplectic structure on 
$G$, i.e.,\ a K-cosymplectic structure on $\fg$. Set $\eta=e^7$, $\xi=e_7$, $\omega=-e^{16}+e^{25}+2e^{34}$ and let $g$ be the scalar product which makes $\{e_1,\ldots,e_7\}$ orthonormal. From this 
we 
recover $\phi$. Hence $(\phi,\eta,\xi,g)$ is a K-cosymplectic structure on $S$. We proceed to show
\begin{enumerate}
 \item $b_1(S)=2$, hence $S$ is not coK\"ahler;
 \item $S$ is not the product of a $6$-dimensional manifold and a circle.
\end{enumerate}
Notice that since $G$ is not completely solvable, we can not use Hattori's theorem to compute the cohomology of $S$. We will instead regard $S$ as the mapping torus $N_{\Phi}$ and use 
the following result (see \cite[Lemma 12]{BFM} for a proof):
\begin{lemma}\label{cohomology_mapping_torus}
Let $M$ be a smooth, $n$-dimensional manifold, let $\varphi\colon M\to M$ be a diffeomorphism and let $M_\varphi$ be the corresponding mapping torus. For every $0\leq 
k\leq n$, set $V^k=\ker(\varphi^*-\mathrm{Id}\colon H^k(M;\bR)\to H^k(M;\bR))$ and $C^k=\mathrm{coker}(\varphi^*-\mathrm{Id}\colon H^k(M;\bR)\to H^k(M;\bR))$, where $\varphi^*$ is the map induced 
by $\varphi$ on $H^k(M;\bR)$. Then
\begin{equation}\label{cohom_map_torus}
 H^k(M_\varphi;\bR)\cong V^k \oplus[\eta]\wedge C^{k-1},
\end{equation}
where $\eta$ is the pullback of the generator of $H^1(S^1;\bR)$ under the mapping torus projection $M_\varphi\to S^1$.
\end{lemma}

According to \cite[Chapter 3, Theorem 3.8]{Oprea_Tralle}, the action of 
$\Phi=\Phi_{\frac{\pi}{2}}$ on $H^k(N;\bR)$ is obtained by taking the action of $\varphi_{\frac{\pi}{2}}$ 
on $\bigwedge^k\fh^*$ and considering the action induced on the cohomology $H^k(\fh;\bR)$. We are implicitly identifying the cohomology of $N$ with the cohomology of $\fh$. We can do this thanks to 
Nomizu's theorem. Applying \eqref{cohom_map_torus} with $k=1$, one computes that
\[
 H^1(S;\bR)=\langle [e^3],[e^7]\rangle.
\]
Hence $b_1(S)=2$; this proves $(1)$ above. Using again \eqref{cohom_map_torus}, one computes the remaining Betti numbers of $S$ to be $b_2(S)=3$ and $b_3(S)=7$. Assume now that $S$ is a product, 
$S=P\times S^1$. Notice that $P$ is an aspherical manifold and that $\pi_1(P)$ is solvable, being a subgroup of $\Gamma=\pi_1(S)$. Also, by applying the K\"unneth formula, we obtain that 
$\chi(P)=-1$. But this contradicts Proposition \ref{product} below. This proves $(2)$.
\end{example}

\begin{proposition}\label{product}
 Let $M$ be a compact aspherical manifold whose fundamental group is solvable. Then $\chi(M)=0$.
\end{proposition}

\begin{proof}
 Set $\Gamma=\pi_1(M)$. We notice that $M$ is an Eilenberg-MacLane space $K(\Gamma,1)$. By a result of Bieri (see \cite[Theorem 9.23]{Bi}), $\Gamma$ is torsion-free and polycyclic. In \cite[Theorem 
1.2]{Baues}, Baues has constructed an infra-solvmanifold $M_\Gamma$ with $\pi_1(M_\Gamma)\cong\Gamma$. In particular, it follows that $M_\Gamma$ is an Eilenberg-MacLane space $K(\Gamma,1)$ and that 
$\chi(M_\Gamma)=0$. Since $M$ is also a $K(\Gamma,1)$, one deduces that $\chi(M)=0$.
\end{proof}

\subsection*{Acknowledgments}
The authors would like to thank J. Oprea and S. Rollenske for useful conversations. G. Bazzoni is partially supported by SFB 701 - Spectral Structures and Topological Methods in Mathematics. M. 
Fern\'andez is partially supported by (Spanish) MINECO grant MTM2011-28326-C02-02 and Project UPV/EHU ref.\ UFI11/52. V. Mu\~noz is partially supported by MINECO grant MTM2012-30719.

\end{document}